\documentclass{dk_article}
\usepackage[top=3cm, bottom=3cm, left=3cm, right=3cm]{geometry}
\usepackage{graphicx}
\usepackage{amsmath}
\usepackage{todonotes}
\usepackage{amsthm}
\usepackage{mhequ}
\usepackage{color}

\usepackage{calrsfs}
\DeclareMathAlphabet{\pazocal}{OMS}{zplm}{m}{n}

\usepackage{hyperref}
\usepackage{amssymb}

\definecolor{darkred}{rgb}{.7,0,0}

\definecolor{darkgreen}{rgb}{0,0.7,0}

\definecolor{darkblue}{rgb}{0,0,0.7}

\title{Deterministic homogenization for fast-slow systems with chaotic noise}
\author{David Kelly$^{\star}$ \quad Ian Melbourne$^{\dagger}$}
\institute{$\star$\;Courant Institute, New York University, NY, USA.\\ \email{dtkelly@cims.nyu.edu}\\ \;$\dagger$\;\;Mathematics Institute, University of Warwick, Coventry CV4 7AL, UK. \email{I.Melbourne@warwick.ac.uk} }

\newcommand{\naturals}{\mathbb{N}}

\newcommand{\ahat}{\tilde{a}}

\newcommand{\textand}{\quad\text{and}\quad}
\newcommand{\xepsR}{x_{\eps,R}}

\newcommand{\Weps}{W_\eps}
\newcommand{\Veps}{V_\eps}

\newcommand{\xeps}{x_{\eps}}

\newcommand{\atilde}{\tilde{a}}

\newcommand{\edist}{\stackrel{dist}{=}}

\newcommand{\floor}[1]{\lfloor #1 \rfloor}

\newcommand{\CA}{\pazocal{A}}

\newcommand{\yeps}{y_{\eps}}

\newcommand{\auto}{\mathfrak{B}}
\newcommand{\diffu}{\mathfrak{A}}
\newcommand{\anti}{\mathfrak{D}}

\newcommand{\rbar}{{\bar{r}}}

\newcommand{\BE}{\mathbf{E}}

\newcommand{\BP}{\mathbf{P}}

\newcommand{\gen}{\pazocal{L}}

\newcommand{\BW}{\mathbf{W}}
\newcommand{\BBW}{\mathbb{W}}

\newcommand{\Wtilde}{\widetilde{W}}

\newcommand{\vhat}{\tilde{v}}

\newcommand{\what}{\tilde{w}}

\newcommand{\normy}[1]{|\!|\!| #1 |\!|\!|}

\newcommand{\CB}{\pazocal{B}}

\newcommand{\CF}{\pazocal{F}}

\newcommand{\BBB}{\mathbb{B}}

\newcommand{\Aeps}{A^{(\eps)}}
\newcommand{\R}{{\mathbb R}}

\newcommand{\CC}{\mathcal{C}}

\newcommand{\abar}{\bar{a}}

\newcommand{\mnu}{\mu}
\newcommand{\BPmnu}{\mnu}

\begin{document}
\maketitle
\large

\begin{abstract}Consider a fast-slow system of ordinary differential equations of the form $\dot x=a(x,y)+\eps^{-1}b(x,y)$, $\dot y=\eps^{-2}g(y)$, where it is assumed that $b$ averages to zero under the fast flow generated by $g$.  We give conditions under which solutions $x$ to the slow equations converge weakly to an It\^o diffusion $X$ as $\eps\to0$.  The drift and diffusion coefficients of the limiting stochastic differential equation satisfied by $X$ are given explicitly.

Our theory applies when the fast flow is Anosov or Axiom~A, as well as to a large class of nonuniformly hyperbolic fast flows (including the one defined by the well-known Lorenz equations), and our main results do not require any mixing assumptions on the fast flow.
\end{abstract}

\section{Introduction}\label{s:intro}
Let $\{\phi_t\}_{t\geq 0}$ be a smooth, deterministic flow on a finite dimensional manifold $M$, with invariant ergodic probability measure $\mnu$. One should think of $\phi_t$ as the flow generated by an ordinary differential equation (ODE) with a chaotic invariant set $\Omega\subset M$ and $\mnu$ supported on $\Omega$.  Define $y(t) = \phi_t y_0$ where the initial condition $y_0$ is chosen at random according to $\mnu$. Hence $y(t) = y(t, y_0)$ is a random variable on the probability space $(\Omega,\mnu)$; from here on we omit $y_0$ from the notation, as is conventional with random variables. Let $a,b : \reals^d \times M \to \reals^d$ be vector fields with suitable regularity assumptions. We are interested in the asymptotic behaviour of the ODE
\begin{equ}
\frac{dx^{(\eps)}}{dt} = \eps^2 a(x^{(\eps)}\!,y) + \eps b(x^{(\eps)}\!,y)  \quad \;, \quad x^{(\eps)}(0) = \xi
\end{equ}
as $\eps \to 0$ and $t \to \infty$, with $\eps^2 t$ remaining fixed. The initial condition $\xi \in \reals^d$ is assumed deterministic. Due to the dependence on $y_0$, we interpret $x^{(\eps)}$ as a random variable on $\Omega$ taking values in the space of continuous functions $C([0,T], \reals^d)$ for some finite $T>0$. 
\par
To make the statement of convergence precise, we define $\yeps(t) = y(\eps^{-2}t)$ and $\xeps$ as the solution to the ODE
 \begin{equ} \label{e:fastslow}
\frac{d\xeps}{dt} =  a(\xeps,\yeps) + \frac{1}{\eps}b(\xeps,\yeps) \quad , \quad x_\eps(0) = \xi\;.
\end{equ}
In particular, we arrive at this equation under the rescaling $t \mapsto t/\eps^2$ and setting $\xeps(t) = x^{(\eps)}(t/\eps^2)$. Our aim is to identify the limiting behavior of the random variable $\xeps$ on the space of continuous functions as $\eps \to 0$. 
\par
Under certain assumptions on the fast flow $\phi_t$, it is known that $x_\eps \to_w X$ where $X$ is an It\^o diffusion, and where $\to_w$ denotes weak convergence of random variables on the space $C([0,T], \reals^d)$. 
At an intuitive level, 
the $a$ term \emph{averages} to an ergodic mean, via a law of large numbers type effect and 
the $b$ term \emph{homogenizes} to a stochastic integral, via a central limit theorem type effect. This type of problem is often referred to as \emph{deterministic} homogenization, since the randomness is not coming from a typical stochastic process, but rather from a ergodic dynamical system with random initial condition. 
\par 
Assuming rather strong mixing conditions on $\phi_t$, one can show that $x_\eps$
converges weakly to the solution $X$ of an It\^o SDE
\begin{equ}\label{e:papa_sde}
dX = \ahat (X) dt + \sigma(X) dB \quad \;, \quad X(0) = \xi
\end{equ}
where $B$ is an $\reals^d$ valued standard Brownian motion, the drift $\ahat : \reals^d \to \reals^d$ is given by 
\begin{equ}
\ahat^i(x) = 
\int_\Omega a^i(x,y)d\mnu(y) +
\int_0^\infty \int_\Omega b(x,y)\cdot\nabla b^i(x,\phi_ty)d\mnu(y)dt  
\end{equ} 
for all $i = 1,\dots, d$ and the diffusion coefficient $\sigma : \reals^d \to \reals^{d\times d}$ is given by 
 \begin{equ}
\sigma (x) \sigma^T(x) = \int_0^\infty\int_\Omega 
\bigl(b(x,y)\otimes b(x,\phi_t y)+ (b(x,\phi_t y)\otimes b(x,y)\bigr)d\mnu(y)dt \;.
 \end{equ}
The mixing assumptions required on $\phi_t$ are typically very strong. For instance, the above result follows from \cite{papa74} under the assumption of phi mixing with rapidly decaying mixing coefficient ($L^{1/2}$-integrable). Such an assumption is quite reasonable in the setting of ergodic Markov processes (as intended in \cite{papa74}). Unfortunately this is quite unreasonable for general ergodic flows. In particular, for most natural \emph{deterministic} situations it is difficult to prove any mixing properties at all. On top of that, it is seldom clear that the formulas for $\ahat$ and $\sigma$ are even well-defined. 

\par
In this article, we show that for a very general class of ergodic flows, the above result holds with explicit (but sometimes more complicated) formulas for $\ahat$ and $\sigma$ that generalise the ones given above.

\subsection{Anosov and Axiom A flows}
One well-known class of fast flows to which our results
apply is given by the Axiom A (uniformly hyperbolic) flows introduced by Smale \cite{smale67}.
This includes Anosov flows \cite{anosov67}. We do not give the precise definitions, since they
are not needed for understanding the paper, but a rough description is as follows.
(See \cite{bowen75,ruelle78,sinai70} for more details.)
Let $\phi_t : M \to M $ be a $C^2$ flow defined on a compact manifold $M$. A flow-invariant
subset $\Omega \subset M$ is uniformly hyperbolic if for all $x \in \Omega$ there exists a $D \phi_t$-invariant splitting transverse to the flow into uniformly contracting and expanding directions. The flow is Anosov if the whole of $M$ is uniformly hyperbolic. More generally, an
Axiom A flow is characterised by the property that the dynamics decomposes into
finitely many hyperbolic equilibria and finitely many uniformly hyperbolic subsets
$\Omega_1,\dots ,\Omega_k$, called hyperbolic basic sets, such that the flow on each $\Omega_i$ is transitive (there is a dense orbit). If $\Omega$ is a hyperbolic basic set, there is a unique $\phi_t$-invariant ergodic probability measure (called an equilibrium measure) associated to each H\"older function on $\Omega$. (In the special case that $\Omega$ is an attractor, there is a distinguished equilibrium measure called the physical measure or SRB measure (after Sinai, Ruelle, Bowen).) In the remainder of the introduction, we assume that $\Omega$ is a hyperbolic basic set with equilibrium measure $\mnu$ (corresponding to a H\"older potential). We exclude the trivial case where $\Omega$ consists of a single periodic orbit.

Given $b:\R^d\times M\to\R$, we define the mixed H\"older norm
 \begin{equ}
\norm{b}_{C^{\alpha,\kappa}}= \sum_{|k| \leq \floor{\alpha}} \sup_{x \in \reals^d} \norm{D^k b(x,\cdot)}_{C^{\kappa}} + \sum_{|k|= \floor{\alpha}} \sup_{x,z \in \reals^d} \frac{\norm{D^k b(x,\cdot) - D^k b(z,\cdot)}_{C^\kappa}}{|x-z|^{\alpha-\floor{\alpha}}}
\end{equ}
for $\alpha\in[0,\infty)$, $\kappa\in[0,1)$, where the second summation is omitted when $\alpha$ is an integer.
Here $D^k$ is the differential operator acting in the $x$ component and $\norm{\cdot}_{C^\kappa}$ is the standard H\"older norm acting in the $y$ component.
If $b:\R^d\times M\to\R^m$ is vector-valued, we define
$\norm{b}_{C^{\alpha,\kappa}}=\sum_{i=1}^m\norm{b^i}_{C^{\alpha,\kappa}}$.

We write $b\in C^{\alpha,\kappa}(\R^d\times M,\R^m)$ if $\norm{b}_{C^{\alpha,\kappa}}<\infty$.
Let $C^\kappa_0(\R^d\times M,\reals^m)$ denote the space of observables $b\in C^{\alpha,\kappa}(\R^d\times M,\R^m)$
with $\int_\Omega b(x,y)d\mnu(y)=0$ for all $x\in\R^d$. When $m=1$, we write $C^{\alpha,\kappa}(\R^d\times M)$ 
instead of $C^{\alpha,\kappa}(\R^d\times M,\R^m)$ and so on.
We also write $C_0^\kappa(\Omega,\R^m)$ to denote $C^\kappa$ observables
$v:\Omega\to\R^m$ with mean zero. We now state the main result. 
 
\begin{thm}\label{thm:axioma}
Let $\Omega \subset M$ be a hyperbolic basic set with equilibrium measure $\mnu$.
Let $\kappa>0$ and suppose 
that $a\in C^{1+,0}(\reals^d\times M,\reals^d)$, $b\in C^{2+,\kappa}_0(\reals^d\times M,\reals^d)$.
Then 
\begin{itemize}
\item[(i)] The limit 
\[
\auto(v,w)=\lim_{n\to\infty}n^{-1}\int_\Omega \int_0^n\int_0^s v\circ\varphi_rw\circ\varphi_sdrds,
\]
exists for all $v,w\in C_0^\kappa(\Omega)$ and the resulting bilinear
operator $\auto:C_0^\kappa(\Omega) \times C_0^\kappa(\Omega)\to\R$
is bounded and positive semidefinite.
\item[(ii)]  The drift and diffusion coefficients given by
\begin{equ}
\ahat^i(x) = \int a^i(x,y)d\mnu(y) +  \sum_{k=1}^d \auto(b^k(x,\cdot),\del_k b^i(x,\cdot))  \quad, \quad i=1,\dots,d\;, 
\end{equ}
\begin{equ}
(\sigma (x) \sigma^T(x))^{ij} = \auto(b^i(x,\cdot),b^j(x,\cdot)) + \auto(b^j(x,\cdot),b^i(x,\cdot)) \quad, \quad  
i,j =1, \dots, d \;,
 \end{equ}
are Lipschitz.
\item[(iii)] 
The family of solutions $\xeps$ to the ODE~\eqref{e:fastslow} converges weakly in the supnorm 
topology to the unique solution $X$ of the SDE
\begin{equ}\label{e:sde}
dX = \ahat(X)dt + \sigma(X)dB \quad , \quad X(0) = \xi
\end{equ}
where $B$ is a standard Brownian motion in $\R^d$.
\item[(iv)] Let $v,w\in C^\kappa_0(\Omega)$.  If in addition the integral $\int_0^\infty \int_\Omega v\, w\circ \phi_t   d\mnu dt $ exists, then 
\begin{equ}
\auto (v,w) = \int_0^\infty \int_\Omega v \,w\circ \phi_t d\mnu dt \;.
\end{equ}
\end{itemize}
\end{thm}

By part (i), the $\auto$-terms in the formulas for $\tilde a$ and $\sigma$ can be written as
\begin{equ}
\auto(b^k(x,\cdot),\del_k b^i(x,\cdot)) = \lim_{n\to\infty} n^{-1} \int_\Omega \int_0^{n} \int_0^s b^k(x,\phi_r y) \del_k b^i(x,\phi_s y) dr ds  d\mnu(y)
\end{equ}
and 
\begin{equ}
\auto(b^i(x,\cdot),  b^j(x,\cdot)) = \lim_{n\to\infty} n^{-1} \int_\Omega \int_0^{n} \int_0^s b^i(x,\phi_r y)  b^j(x,\phi_s y) dr ds  d\mnu(y)\;.
\end{equ}
By part (iv), if the integrals $\int_0^\infty \int_\Omega b^i(x,y) b^j(x,\phi_t y)  d\mnu(y) dt $, $\int_0^\infty \int_\Omega b^k(x,y) \del_k b^i(x,\phi_t y)  d\mnu(y) dt $ exist for all $i,j,k$ and $x\in\reals^d$, then the coefficients $\tilde a$ and $\sigma$ are given by the formulas in \eqref{e:papa_sde}.  In general, even for nonmixing flows $\phi_t$, the bilinear operator $\auto$ can still be written down explicitly, in terms of the finer structure of the flow, see~\eqref{eq-B}.
\begin{rmk}
The Lipschitz statement in Theorem~\ref{thm:axioma}(ii) follows from boundedness of $\auto$ together with the regularity assumptions on $a$ and $b$.  A consequence of this is the uniqueness of the limiting diffusion $X$ as stated in part (iii).
\end{rmk}

\begin{rmk}
Since the expression defining $\sigma\sigma^T$ is symmetric and positive semidefinite,
a square root $\sigma$ always exists. Also, it is a standard result that the diffusion $X$ is independent of the choice of any square root $\sigma$. 
\end{rmk}

\begin{rmk}\label{rmk:sigma_D}
In the special case $b(x,y)=h(x)v(y)$ considered 
in \cite{kelly14}, $\auto(v^i,v^j)$ is defined through its symmetric part, denoted $\frac{1}{2}\Sigma^{ij}$ and its anti-symmetric part, denoted $\frac{1}{2}D^{ij}$. One easily recovers the above via the It\^o-Stratonovich correction. 
\end{rmk}

\subsection{Non-uniformly hyperbolic flows}
For the sake of exposition, we have stated the homogenization results for
fast flows that are uniformly hyperbolic. In reality, the results apply much more generally. The convergence result stated in Theorem \ref{thm:axioma}(iii) can be recast in an abstract framework. In brief, we only require that $\phi_t$ satisfies an {iterated central limit theorem} (CLT) along with a moment control estimate. As shown in \cite{kelly14}, these assumptions hold true for broad classes of flows which have a Poincar\'e map modelled by a Young tower \cite{young98,young99}. For example, the convergence result in Theorem \ref{thm:axioma}(iii) holds for the classical Lorenz equations. We provide a rigorous statement of the abstract formulation in Section \ref{s:fast}.

\subsection{Previous results}

It is only fairly recently that results on homogenization have been obtained in a fully deterministic setting with realistic assumptions on the fast dynamics.  The first such results were obtained by~\cite{dolgopyat04,dolgopyat05} for discrete time systems where the fast dynamics is uniformly or partially hyperbolic with sufficiently fast decay of correlations.  A program to remove assumptions on decay of correlations on the fast dynamics was initiated in~\cite{stuart11} where the authors prove a result on homogenization for uniformly and nonuniformly hyperbolic flows that are not necessarily mixing, but under the assumption that the noise appears additively in the slow ODE, that is $b(x,y) = h(y)$.   This was extended to the case of multiplicative noise
$b(x,y) = h(x) v(y)$ in the scalar case $d=1$ by \cite{gottwald13} who also
treated the discrete time situation.   
The case $b(x,y)=h(x)v(y)$ was treated in general dimensions in \cite{kelly14} (again for both discrete and continuous time).
We remark that the results of the current article should carry over to the discrete time setting, but this requires additional work to incorporate the discrete time rough path theory introduced in \cite{kelly14a}.
\par
Homogenization results for chaotic systems have many interesting physical applications, most notably in stochastic climate models \cite{majda99}. For more examples, see \cite[Section 11.8]{stuart08}. 

\subsection{Outline of the article}\label{sec:strategy}

To prove Theorem \ref{thm:axioma} (or more precisely Theorem \ref{thm:abstract}, the abstract version) we reformulate the slow equation as an ODE of the form
\begin{equ}
d\xeps = F(\xeps)d\Veps + H(\xeps)d\Weps \;,
\end{equ}
where $\Veps$ and $\Weps$ are function space valued paths that are smooth (in time) for each fixed $\eps$. The path $\Veps$ is a smooth approximation of a function space valued drift and the path $\Weps$ is a smooth approximation of a function space valued Brownian motion. To be precise, we take
\begin{equ}
\Veps(t) = \int_0^t a(\cdot, \yeps(r))dr \textand \Weps(t) = \eps^{-1}\int_0^t b(\cdot, \yeps(r))dr\;.
\end{equ}
The operators $F(x), H(x)$ are Dirac distributions (evaluation maps) located at $x$, that is $F(x) \vphi = \vphi(x)$ for any $\vphi $ in the function space and similarly for $H$. 
\begin{rmk}
Note that although $F,H$ are both Dirac distributions, they will act on different domains, hence the different labels.\end{rmk}
One should think of the pair $(\Veps,\Weps)$ as ``noise'' driving the solution $\xeps$. Using the theory of rough paths, we build a continuous solution map from the ``noise space'' into the ``solution space''. The ``noise space'' will contain not just smooth paths, but also paths of Brownian regularity (which is the type of regularity we expect from the limiting $W_\eps$). Since the solution map is continuous, a weak convergence result for the noise processes can be lifted to a weak convergence result for the solution, via the continuous mapping theorem. 
\par
The outline of the article is as follows. In Section \ref{s:fast} we write the abstract formulation of Theorem \ref{thm:axioma}; this constitutes the main result of the article. In Section \ref{s:rough} we give an overview of rough path theory and state the tools that will be used. In Sections \ref{s:convergence}, \ref{s:conv_to_RDE} and \ref{s:rde_ito} we state and prove a localized version of the main result. In Section \ref{s:localization} we lift the localized result to the full result.

\subsection{Notation}

We write $\BE_\mnu$ for expectation with respect to $\mnu$ and write $\BE$ when referring to expectation on a generic probability space. 
We write for example $a\in C^{1+}$ if there exists $\alpha>1$ such that
$a\in C^\alpha$.
For a normed linear space $\CB$ we write $L(\CB,\reals)$ for the space of bounded linear functionals on $\CB$, with the usual norm $\norm{f}_{L(\CB,\reals)} = \sup_{ \norm{x}_{\CB} = 1} |f (x)|$.  
We write
$a_n\lesssim b_n$ as $n\to\infty$ if there is a constant
$C>0$ such that $a_n\le Cb_n$ for all $n\ge1$.

\section{The abstract convergence result}\label{s:fast}

We now state an abstract version of Theorem \ref{thm:axioma}. 
Let $\phi_t:M\to M$ be a smooth flow on a finite dimensional manifold and suppose that 
$\Omega\subset M$ is a closed flow-invariant set with ergodic probability measure $\mnu$.
For $v\in L^1(\Omega,\reals^m)$ with $\int_\Omega v d\mnu = 0$, we define
\begin{equ}
W_{v,n}(t)=n^{-1/2}\int_0^{tn} v\circ \phi_s\,ds \textand
\BBW_{v,n}(t)=n^{-1}\int_0^{tn}\int_0^s v\circ \phi_r \otimes v\circ \phi_s drds\;.
\end{equ}
By definition of the tensor product for vectors, $\BBW_{v,n}$ takes values in $\reals^{m\times m}$. 

For $v,w\in L^1(\Omega,\reals)$, we define
\begin{equ}
v_t =  \int_0^t v \circ \phi_s ds \textand S_t = \int_0^t \int_0^s v\circ \phi_r w \circ \phi_s drds\;. 
\end{equ}

Fix $\kappa > 0$. 
The abstract assumptions are as follows. 

\begin{ass}\label{ass:wip}\label{ass:clt} There exists a bilinear operator $\auto : C^\kappa_0(\Omega) \times C^\kappa_0(\Omega) \to \reals$ such that for every $v \in C^\kappa_0 (\Omega, \reals^m)$,
\begin{equ}
(W_{v,n},\BBW_{v,n}) \to (W_v,\BBW_v)
\end{equ}
as $n\to\infty$, in the sense of finite dimensional distributions, where $W_v$ is a Brownian motion in $\reals^m$ and $\BBW_v$ is the process with values in $\R^{m\times m}$ defined by 
\begin{equ}
\BBW_v^{ij}(t) = \int_0^t W_v^i dW_v^j + \auto(v^i , v^j) t\;.
\end{equ}
(Here, the integral is of It\^o type.)
\end{ass}

\begin{ass}\label{ass:moments}
There exists $p >3$, and for all $v,w\in C^\kappa_0(\Omega)$  there exists \mbox{$K=K_{v,w,p}>0$} such that 
\begin{equ}
(\BE_\mnu |v_t|^{2p})^{1/(2p)} \le K t^{1/2} \textand (\BE_\mnu |S_t|^{p})^{1/p} \le K t
\end{equ}
for all $t\ge0$. If the estimates hold for all $p > 3$ then we say the estimates hold for $p=\infty$. 
\end{ass}

\begin{thm}\label{thm:abstract}
Suppose that Assumptions \ref{ass:clt} and \ref{ass:moments} hold with some $p \in (3,\infty]$ and $\kappa>0$. 
Suppose that 
$a\in C^{1+,0}(\reals^d\times M,\R^d)$ and $b\in C^{\alpha,\kappa}_0(\reals^d\times M,\R^d)$ for some $\alpha>2+\frac{2}{p-1}+\frac{d}{p}$.
Then we have the same conclusion as Theorem \ref{thm:axioma}(i,ii,iii).
\end{thm}
We now show how Theorem~\ref{thm:axioma} follows from Theorem~\ref{thm:abstract}.

\begin{proof}[Proof of Theorem~\ref{thm:axioma}]
Assumptions~\ref{ass:clt} and~\ref{ass:moments} (with $p=\infty$) are valid for hyperbolic basic sets by~\cite[Theorem~1.1]{kelly14} and~\cite[Proposition~7.5, Remark~7.7]{kelly14} respectively.
Hence Theorem~\ref{thm:axioma}(i,ii,iii) follows from Theorem~\ref{thm:abstract}.
Moreover, Theorem~\ref{thm:axioma}(iv) follows from~\cite[Theorem~1.1(b)]{kelly14}. 
\end{proof}

\begin{rmk} \label{rmk-KM}
In~\cite{kelly14}, we considered the special case where $b(x,y)=h(x)v(y)$ is a product (for some $v:M\to\R^e$ and $h:\R^d\to \R^{d\times e}$) under less stringent regularity conditions on $b$.   It is easy to check that when $b$ is a product, the method in this paper applies provided $b\in C^{\alpha,\kappa}$ for some
$\alpha>2+2/(p-1)$ recovering the results of~\cite{kelly14}.
The only place where the additional regularity is required for general $b$
is in the tightness
estimates in Section~\ref{s:conv_to_RDE} below.
\end{rmk}

\begin{rmk} \label{rmk-B}
A general formula for the bilinear operator $\auto$ in the case of (not necessarily mixing) Axiom~A flows can be obtained by considering the associated suspension flow.   We recall the basic definitions; further details can be found in~\cite{kelly14} and references therein.

Suppose that $f:\Lambda\to\Lambda$ is a map with ergodic invariant probability measure $\mu$.
Let $r:\Lambda\to\R^+$ be an integrable roof function with $\bar r=\int_\Lambda r\,d\mu$.   
Define the suspension
$\Lambda^r=\{(x,u)\in\Lambda\times\R:0\le u\le r(x)\}/\sim$ where
$(x,r(x))\sim (fx,0)$,  Define the suspension flow $\phi_t(x,u)=(x,u+t)$ computed
modulo identifications.   The measure $\mu_\Lambda^r=\mu_\Lambda\times{\rm Lebesgue}/\bar r$ is an ergodic invariant probability measure for $\phi_t$.

Every hyperbolic basic set $(\Omega,\mnu)$ for an Axiom~A flow can be identified with a  suspension $(\Lambda^r,\mu_\Lambda^r)$ 
with continuous and bounded roof function $r$ over a mixing subshift of finite type $(\Lambda,\mu_\Lambda)$.

Given $v\in C^\kappa_0(\Omega,\reals^m)$, we define the induced observable
$\tilde v\in L^\infty(\Lambda,\mu_\Lambda)$ by setting
$\tilde v=\int_0^r v\circ \phi_tdt$.
Similarly, we associate $\tilde w$ to $w$.
It can be shown that $\tilde v$ and $\tilde w$ have exponential decay of correlations, so in particular the series $\sum_{n=1}^\infty \int_\Lambda \tilde v\,\tilde w\circ f^n\,d\mu_\Lambda$ is absolutely convergent.

Moreover, as shown in~\cite[Corollary~8.1]{kelly14}, 
\begin{equ} \label{eq-B}
\auto(v,w)=(\bar r)^{-1}\sum_{n=1}^\infty\int_\Lambda \tilde v\,\tilde w\circ f^n\,d\mu_\Lambda+(\bar r)^{-1}\int_\Lambda S(v,w)d\mu_\Lambda\;,
\end{equ}
where 
\begin{equ}
S(v,w)(y) = \int_0^{r(y)}\bigg(  \int_0^{s} v(\phi_u y) du \bigg) w (\phi_s y) ds \;.
\end{equ}
is the iterated integral of the path $(v\circ \phi_t,w\circ \phi_t)$ along the orbit until its return to $\Lambda$. 
\end{rmk}

\begin{rmk}
There is a slightly simpler way of writing $\auto$ which gives a more geometric description of the bilinear form. We introduce the symmetric and anti-symmetric parts 
\begin{equ}
\diffu (v,w) = \frac{1}{2}\bigg(\auto(v,w) + \auto(w,v)\bigg)  \textand \anti (v,w) = \frac{1}{2}\bigg(\auto(v,w) - \auto(w,v) \bigg)\;.
\end{equ}
For the symmetric part, it follows from the product rule that 
\begin{equ}
S(v,w) + S(w,v) = \bigg(\int_0^r v\circ \phi_t dt\bigg) \bigg( \int_0^r w\circ \phi_t  dt \bigg)  = \vhat \what
\end{equ} 
and hence
\begin{equ}
\diffu (v,w) = \sum_{n=1}^\infty  \frac{1}{2\rbar}\int_\Lambda \big(\vhat\, \what\circ f^n +\what\,\vhat\circ f^n \big) d\mu_\Lambda + \frac{1}{2\rbar } \int_\Lambda \vhat\what d\mu_\Lambda\;.
\end{equ}
In particular the symmetric part of the bilinear form is completely determined by the cross correlations between induced observables. Similarly 
\begin{equ}
\anti (v,w) =  \sum_{n=1}^\infty  \frac{1}{2\rbar}\int_\Lambda \big(\vhat\, \what\circ f^n -\what\, \vhat\circ f^n\big) d\mu_\Lambda + \frac{1}{2\rbar } \int_\Lambda \big( S(v,w) - S(w,v)\big) d\mu_\Lambda\;.
\end{equ}
The advantage here is that the expression
\begin{equ}
\frac{1}{2}\big( S(v,w)(y) - S(w,v)(y)\big) 
\end{equ}
is equal (by Green's theorem) to the \emph{signed area} traced out in $\reals^2$ by the loop $(v(\phi_t y), w(\phi_t y))_{t=0}^{r(y)}$ (closed by the secant joining the endpoints). 
\end{rmk}

We have shown that Theorem~\ref{thm:axioma} holds for Anosov and Axiom A flows, with \mbox{$p = \infty$}. More generally, the conclusions remain valid for nonuniformly hyperbolic flows with Poincar\'e map modelled by a Young tower with exponential tails~\cite{young98} (including the case of H\'enon-like attractors). 
Even more generally, the conclusions remain valid 
 when the Poincar\'e map is modelled by a Young tower with subexponential tails~\cite{young99} provided that the tails decay sufficiently quickly
(the value of $p$ depends on this decay rate).  We refer to \cite[Section~10]{kelly14} for a precise statement.
In particular, this includes the classical Lorenz attractor (again
with $p=\infty$).
\par
In the remainder of this section, we describe some elementary
properties that follow immediately from the assumptions on the fast dynamics. Firstly, we show that in Assumption~\ref{ass:moments} the constant $K$ can be chosen uniformly in $v,w$. Define the incremental objects
\begin{equ}
v_{s,t} =  \int_s^t v \circ \phi_r dr \textand S_{s,t} = \int_s^t \int_s^r v\circ \phi_u w \circ \phi_r dudr\;. 
\end{equ}

\begin{prop} \label{prop:moments}  
If the fast flow satisfies Assumption \ref{ass:moments}, 
then 
 \begin{equ}
 (\BE_\mu |v_{s,t}|^{2p})^{1/(2p)} \lesssim \norm{v}_{C^\kappa} |t-s|^{1/2} \textand (\BE_\mu |S_{s,t}|^{p})^{1/p} \lesssim \norm{v}_{C^\kappa}\norm{w}_{C^\kappa} |t-s|
 \end{equ}
 for all $s,t\ge0$, $v,w\in C^\kappa_0(\Omega)$. 
\end{prop}

\begin{proof}
By stationarity it suffices to check the claim with $s =0$. Consider the family of linear operators $\{L_t:C^\kappa_0(\Omega)\to L_{2p}(\Omega),\;t>0\}$ given by $L_tv=t^{-1/2}v_t$.
Since $\norm{L_tv}_\infty\le t^{1/2}\norm{v}_\infty$ it is certainly the case that 
$L_t:C^\kappa_0(\Omega)\to L_{2p}(\Omega)$ is bounded for each $t$.
By Assumption \ref{ass:moments}, for each $v\in C^\kappa_0(\Omega)$, there
exists a constant $K=K_v$ such that $\norm{L_tv}_{2p}\le K_v$ for all $t>0$.
By the uniform boundedness principle, there is a uniform constant $K$ such that
$\norm{L_tv}_{2p}\le K\norm{v}_{C^\kappa}$ for all $v\in C^\kappa_0(\Omega)$, $t>0$.
This establishes the desired estimate for $v_t$.

The estimate for $S_t$ is proved similarly by considering the family of
bilinear operators $\{B_t:C^\kappa_0(\Omega)\times C^\kappa_0(\Omega)\to L^p(\Omega),\;t>0\}$ given by $B_t(v,w)=t^{-1}S_t$.
\end{proof}

The next result is a collection of simple facts that will be used throughout the rest of the article.  

\begin{prop} \label{prop:auto}  
If the fast flow satisfies Assumptions \ref{ass:wip} and \ref{ass:moments}, then 
\begin{itemize}
\item[(a)] The covariance of $W_v$ is given by
$ \BE W_v^i(1)W_v^j(1)=\auto(v^i,v^j)+\auto(v^j,v^i)$
for all $v\in C^\kappa_0(\Omega,\reals^m)$.
\item[(b)] $\auto(v,v)\ge0$ for all
$v\in C^\kappa_0(\Omega)$.
\item[(c)] $|\auto(v,w)|\lesssim\norm{v}_{C^\kappa}\norm{w}_{C^\kappa}$ for all
$v,w\in C^\kappa_0(\Omega)$.
\item[(d)] $(W_{v,n},\BBW_{v,n})\to_w(W_v,\BBW_v)$ as $n\to\infty$ in the supnorm topology.
\end{itemize}
\end{prop}

\begin{proof}
(a)
It follows from Assumptions~\ref{ass:clt} and~\ref{ass:moments} 
that
\[
\BE_\mu W_{v,n}^i(1) W_{v,n}^j(1)\to \BE W_v^i(1) W_v^j(1),
 \quad
\BE_\mu\BBW_{v,n}^{ij}(1)\to \BE \BBW_v^{ij}(1)= \auto(v^i,v^j),
\]
where we have used the fact that It\^o integrals have zero mean.
Taking expectations on both sides of the identity
\[
W_{v,n}^i(1) W_{v,n}^j(1)= \BBW_{v,n}^{ij}(1)+\BBW_{v,n}^{ji}(1)
\]
and letting $n\to\infty$ yields the desired result.

\noindent(b)
It follows from part~(a) that $\auto(v^i,v^i)=\frac12\BE W_v^i(1)^2\ge0$.

\noindent(c) Define $S_t^{ij}$ using the definition of $S_t$ but with $v=v^i$ and $w=w^i$.
We note that $n^{-1}S_n^{ij}=\BBW_{v,n}^{ij}(1)$ and hence by Proposition~\ref{prop:moments},
$\BE_\mu|\BBW_{v,n}^{ij}(1)|
\lesssim \norm{v^i}_{C^\kappa} \norm{v^j}_{C^\kappa}$.
By Assumptions~\ref{ass:clt} and~\ref{ass:moments}, $|\auto(v^i,v^j)|
=\lim_{n\to\infty}|\BE_\mu\BBW_{v,n}^{ij}(1)|
\lesssim \norm{v^i}_{C^\kappa} \norm{v^j}_{C^\kappa}$.

\noindent(d) Since the limiting random variable is (almost surely) continuous,  it is sufficient to prove the weak convergence result in the Skorokhod topology. But this is a simple consequence of \cite[Theorem 15.6]{billingsley99}, combined with the Assumptions \ref{ass:clt} and \ref{ass:moments}.
\end{proof}

%

Finally, we show that convergence as $n\to\infty$ of the sequence of processes $(W_{v,n},\BBW_{v,n})$ 
implies 
convergence as $\eps\to0$  of the family of processes 
\[
W^{(\eps)}_v(t)= \eps\int_0^{t\eps^{-2}} v\circ \phi_s\, ds, \quad
\BBW^{(\eps)}_v(t)=\eps^2\int_0^{t\eps^{-2}}\int_0^s v\circ\phi_r\otimes v\circ\phi_s\,dr\,ds,\;\eps>0.
\]
Before doing so, we need the following elementary lemma.
 
\begin{lemma} \label{lem-elem}
Suppose that $a:\R\to\R$ is bounded on compact sets.
Let $b>0$, $T\ge0$.
If $\lim_{\eps\to0}\eps^b a(\eps^{-1})=0$, then
$\lim_{\eps\to0}\eps^b\sup_{t\in[0,T]}|a(t\eps^{-1})|=0$.
\end{lemma}

\begin{proof}
The proof is standard and included just for completeness.

Fix $\delta>0$.  Choose $\eps_0>0$ such that
$\eps^b a(\eps^{-1})<\delta/T^b$ for $\eps<\eps_0$.
Now choose $\eps_1>0$ such that
$\eps_1^b\sup_{t\le\eps_0^{-1}}|a(t)|<\delta$.

We show that $\sup_{t\in[0,T]}|\eps^b a(t\eps^{-1})|<\delta$ for all
$\eps<\eps_1$.
There are two cases.
If $\eps/t\ge\eps_0$, then
$|\eps^b a(t\eps^{-1})|\le\eps_1^b\sup_{t\le\eps_0^{-1}}|a(t)|<\delta$.
If $\eps/t<\eps_0$, then
$|\eps^b a(t\eps^{-1})|\le T^b(\eps/t)^b|a(t\eps^{-1})|<\delta$.
\end{proof}

\begin{prop} \label{prop-neps}
If Assumption~\ref{ass:clt} holds, then
$(W^{(\eps)}_v,\BBW^{(\eps)}_v)\to_w(W_v,\BBW_v)$ as $\eps\to0$ in the supnorm topology, for all $v\in C_0^\kappa(\Omega,\reals^m)$.
\end{prop}

\begin{proof}
Let $n=[\eps^{-2}]$. We have
\[
W^{(\eps)}_v(t)=\eps n^{1/2}W_{v,n}(t)+
\eps\int_{tn}^{t\eps^{-2}}v\circ \phi_s\,ds.
\]
As $\eps\to0$, $\eps n^{1/2}\to1$.  Also,
$\|\int_{tn}^{t\eps^{-2}}v\circ \phi_s\,ds\|_\infty\le t\norm{v}_\infty$ and hence $W^{(\eps)}_v-W_{v,n}\to_w0$.

Similarly,
\begin{align*}
\BBW^{(\eps)}_v(t) & =\eps^2 n\BBW_{v,n}(t)+
\eps^2\int_{tn}^{t\eps^{-2}}\int_0^s v\circ \phi_r\otimes v\circ\phi_s\,dr\,ds=
\eps^2 n\BBW_{v,n}(t)+\eps^2 \Aeps(t),
\end{align*}
where
\begin{align*}
\Aeps(t) & = \int_{tn}^{t\eps^{-2}}\int_{tn}^s v\circ \phi_r\otimes v\circ\phi_s\,dr\,ds+
\int_{tn}^{t\eps^{-2}}\int_0^{tn} v\circ \phi_r\otimes v\circ\phi_s\,dr\,ds .
\end{align*}
Now $|\Aeps(t)|\le t^2|v|_\infty^2+t|v|_\infty |v_{tn}|$.
By the ergodic theorem, $\eps^2 v_{\eps^{-2}}\to0$ almost everywhere, and hence by Lemma~\ref{lem-elem} 
$\sup_{t\in[0,T]}\eps^2 |v_{t\eps^{-2}}|\to0$ almost everywhere.
It follows that
$\sup_{t\in[0,T]}\eps^2 |\Aeps(t)|\to0$ almost everywhere, and so
$\BBW^{(\eps)}_v-\BBW_{v,n}\to_w0$.

Altogether, we obtain that
$(W^{(\eps)}_v,\BBW^{(\eps)}_v)-
(W_{v,n},\BBW_{v,n})\to_w 0$ as required.
\end{proof}

\section{Some rough path theory}\label{s:rough}
In this section we formalize some of the ideas from rough path theory put forward in Section~\ref{s:intro}: namely, that one can build a continuous map from ``noise space'' to ``solution space''. This map is constructed using rough path theory \cite{lyons98}. The formulation of rough path theory that we employ closely follows the recent book \cite{frizhairer13}. Before going into the theory, we list some preliminary facts concerning tensor products of Banach spaces.

\subsection{Tensor products of Banach spaces}
Let $\CA,\CB$ be Banach spaces (over $\reals$). The \emph{algebraic} tensor product space $\CA \otimes_{a} \CB$ is defined as the vector space
\begin{equ}
\CA \otimes_{a} \CB = {\rm span }\{ x\otimes y \;|\; x \in \CA \;,\; y\in \CB  \} \;.
\end{equ}
That is, $\CA \otimes_a \CB $ is the space of finite sums $\sum_{n} x_n \otimes y_n $ for $x_n \in \CA$, $y_n \in \CB$. For $f \in L(\CA,\reals)$, $g \in L(\CB,\reals)$ we define a linear functional $f\otimes g : \CA \otimes_a \CB \to \reals$ by 
\begin{equ}\label{e:dual_tensor}
(f\otimes g) \sum_{n} x_n \otimes y_n = \sum_{n} f(x_n) g(y_n)\;.
\end{equ}
A norm $\norm{\cdot}_{\CA \otimes \CB} : \CA \otimes_a \CB \to \reals_+$ is called \emph{admissible} if 
\begin{equ}\label{e:admissible}
\norm{x \otimes y}_{\CA \otimes \CB} = \norm{x}_{\CA} \norm{y}_{\CB} \textand \norm{f\otimes g}_{L(\CA \otimes_a \CB,\reals)} = \norm{f}_{L(\CA,\reals)}\norm{g}_{L(\CB,\reals)}
\end{equ}
for all $x \in \CA , y \in \CB$ and all $f\in L(\CA,\reals)$ and $g \in L(\CB,\reals)$. By \cite[Lemma 1.4]{light85}, to check admissibility it is sufficient to check \eqref{e:admissible} with $=$ replaced by $\leq$. 
\par
For an admissible norm $\norm{\cdot}_{\CA \otimes \CB}$ we define the tensor product space $\CA \otimes \CB$ as the completion of $\CA \otimes_a \CB$ under the norm $\norm{\cdot}_{\CA \otimes \CB}$. Hence $(\CA \otimes \CB , \norm{\cdot}_{\CA \otimes \CB})$ is a Banach space.
All tensor products we consider will be constructed using an admissible norm.
\par
The admissibility requirement guarantees that $f \otimes g \in L(\CA \otimes_a \CB , \reals)$ and since $\CA \otimes_a \CB$ is (by definition) dense in $\CA \otimes \CB$, $f \otimes g$ extends uniquely to an element of $L(\CA \otimes \CB , \reals)$.

\subsection{Spaces of rough paths}\label{s:rough_results}
In this subsection, we show how to build a ``noise space'' of Banach space valued paths as 
mentioned in Section~\ref{sec:strategy}.
Recall that this should include smooth paths and also Brownian-like paths.
It turns out that it is necessary also to add extra structure to the set
of paths.  The resulting space is called the space of {\em rough paths}.
\par
Let $\CA$ be a Banach space. For $\beta \in (\frac1 2,1)$, we define $\mathcal{C}^\beta=\CC^\beta(\CA)$ to be the set of all continuous paths $V : [0,T] \to \CA$ with $V(0) = 0$ and 
\begin{equ}
|V|_{C^\beta} =  \sup_{s,t} \frac{|V(s,t)|}{|t-s|^\beta} < \infty\;,
\end{equ} 
where $V(s,t) = V(t) - V(s)$. The pair $(\CC^\beta,|\cdot|_{C^\beta})$ is a Banach space.
\par
Let $\CB$ be a Banach space with tensor product space $\CB \otimes \CB$. For $\gamma \in (\frac13,\frac12]$, the space $\CC^\gamma=\mathcal{C}^\gamma(\CB)$ is defined to be the set of all continuous paths $(W,\BBW)~:~[0,T]~\to~\CB~\times~(\CB~\otimes~\CB)$ with $(W(0),\BBW(0)) =   0$ and such that 
\begin{equ}
 \sup_{s,t} \frac{\norm{W(s,t)}_\CB}{|t-s|^\gamma } < \infty \textand  \sup_{s,t} \frac{\norm{\BBW(s,t)}_{\CB \otimes \CB}}{|t-s|^{2\gamma} } < \infty\;,
\end{equ}
where $W(s,t) = W(t) - W(s)$ and $\BBW(s,t) = \BBW(t) - \BBW(s) - W(s)\otimes W(s,t) $.
The set $\mathcal{C}^\gamma$ is known as the set of $\gamma$-rough paths and forms a complete metric space under the metric
\begin{equ}
\rho_\gamma((W_1,\BBW_1), (W_2,\BBW_2)) =  \sup_{s,t} \frac{\norm{W_1(s,t) - W_2(s,t)}_\CB}{|t-s|^\gamma } +  \sup_{s,t} \frac{\norm{\BBW_1(s,t) - \BBW_2(s,t)}_{\CB \otimes \CB}}{|t-s|^{2\gamma} } \;.
\end{equ}
We also make use of the norm-like object
\begin{equ}
\normy{(W,\BBW)}_{\CC^\gamma} = \sup_{s,t} \frac{\norm{W(s,t)}_\CB}{|t-s|^\gamma } +  \sup_{s,t} \frac{\norm{\BBW(s,t)}^{1/2}_{\CB \otimes \CB}}{|t-s|^{\gamma} } \;,
\end{equ}
which shows up in some estimates, but does not play any role in defining the topology. 
\par
Finally, we define the set of $(\beta,\gamma)$-rough paths $\CC^{\beta,\gamma}= \mathcal{C}^\beta(\CA)\times \mathcal{C}^\gamma(\CB)$; this is a complete metric space with the product metric.

\begin{rmk}
One should think of $\CC^{\beta,\gamma}$ as the ``noise space''. This space clearly contains irregular Brownian paths, in addition to smooth paths. The pair $\BW=(W,\BBW)$, when combined with the rough path topology, is what we mean by ``extra structure''. 
We view $\BBW(t)$ 
as a candidate for the integral $\int_0^t W \otimes dW$ and $\BBW(s,t)$ as a candidate for $\int_s^t W(s,r)\otimes dW(r)$. Note that since $W$ is only H\"older continuous, there may be many candidates for the integral $\BBW$; hence it must be specified. 
\end{rmk}

Next, we define a subspace $\CC^\gamma_g \subset \CC^{\gamma}$ known as the geometric rough paths. Let $W~:~[0,T] \to \CB$ be a smooth (piecewise $C^1$) path and let $\BBW~:~[0,T] \to \CB\otimes\CB$ be the path of Riemann integrals
\begin{equ}\label{e:iterated}
\BBW(t) = \int_0^t W \otimes dW
= \int_0^t W \otimes \dot W\,dt\;.
\end{equ}
The $\gamma$-geometric rough paths $\CC_g^{\gamma}$ are defined as the closure of the set of all such smooth pairs $(W,\BBW)$ in $\CC^{\gamma}$. 

\begin{rmk}\label{rmk:tensor1}
The smoothness of $W$ combined with 
the admissibility of the tensor product space ensure that 
$t \mapsto W(t)\otimes \dot{W}(t)$  is a piecewise continuous map and hence Riemann integrable. 
\end{rmk}
\subsection{Rough differential equations}

Suppose that $V,W$ are smooth and that $F: \reals^d \to L(\CA,\reals^d)$, $H  : \reals^d \to L(\CB,\reals^d)$. Under suitable regularity assumptions on $F,H$, the ODE
\begin{equ}
X(t) = \xi +  \int_0^t F(X)dV +  \int_0^t H(X)dW
\end{equ}
has a unique solution $X$. We call the map $\Phi : (V,W) \mapsto X$ the solution map. 
In this subsection, we show how the map $\Phi$ extends to the space of rough 
paths $\CC^{\beta,\gamma}$.
\par
For the moment, we suppose that $F$ is $C^1$ and $H$ is $C^2$.  Recall that
$\beta>\frac12$ and $\gamma>\frac13$ and suppose in addition that $\beta + \gamma > 1$.
For $(V,W,\BBW)\in \CC^{\beta,\gamma}$ there is a class of paths $X: [0,T] \to \reals^d$  known as \emph{controlled rough paths} for which one can define the integrals
\begin{equ}
\int_0^t F(X) dV \textand \int_0^t H(X) d\BW\;,
\end{equ}
with the shorthand $\BW = (W,\BBW)$. We call $X$ a controlled rough path if 
$X(s,t)=X(t)-X(s)$ has the form
\begin{equ}
X^i(s,t) = X'_i (s) W(s,t) + O(|t-s|^{2\gamma})
\end{equ} 
for all $i=1,\dots,d$ and $0 \leq s \leq t \leq T$, where $X'_i \in C^\gamma ([0,T] , L(\CB,\reals))$.
For a thorough treatment of controlled rough paths and their use in defining the above integrals, see \cite[Section 4]{frizhairer13}.
\par
Since $\beta+\gamma>1$, the $dV$ integral is well-defined as a Young integral \cite{young36}, namely
\begin{equ}
\int_0^t F(X) dV = \lim_{\Delta \to 0} \sum_{[t_n,t_{n+1}] \in \Delta}  F(X(t_{n}))V(t_{n},t_{n+1})
\end{equ}
where $\Delta = \{ [t_n,t_{n+1}] : 0\leq n \leq N-1 \}$ denotes partitions of $[0,t]$. The integral is defined pathwise, for each fixed $V$.  
\par
The $d\BW$ integral is defined as a compensated Riemann sum
\begin{equ}\label{e:dW_int}
\int_0^t H(X) d\BW =  \lim_{\Delta \to 0} S_{\Delta}
\end{equ}
where
\begin{equ} \label{e:dW_intS}
S^i_{\Delta} = \sum_{[t_n,t_{n+1}] \in \Delta}  H^i(X(t_{n}))W(t_{n},t_{n+1}) +  \sum_{k=1}^d\big( X'_k(t_{n}) \otimes \del_k H^i(X(t_{n})) \big) \BBW(t_{n},t_{n+1}) 
\end{equ}
with $\Delta$ as above. The dual tensor product $X'_i(t_{n}) \otimes \del_k H^i(X(t_{n})) $ is defined as in~\eqref{e:dual_tensor}. Note that the integral is defined pathwise, for each fixed path $(W,\BBW)$.  
In the special case where $W$ is a Brownian path and $\BBW$ is the iterated It\^o integral, $d\BW$ becomes It\^o integration. 

\par
A controlled rough path $X$ is said to solve the RDE $dX = F(X)dV + H(X)d\BW$ with initial condition $X(0)=\xi$ if it solves the integral equation
\begin{equ}
X(t) = \xi +  \int_0^t F(X)dV + \int_0^t H(X) d\BW\;,
\end{equ}
for all $t \in [0,T]$. 
For a thorough treatment of rough differential equations, see \cite[Section~8]{frizhairer13}. 
In particular, we have the following basic result which includes existence, uniqueness and continuous dependence of solutions to RDEs.
%
%

\begin{thm}\label{thm:rde}
Let $\gamma \in (\frac13,\frac12]$ and 
$\gamma'\in (\frac13,\gamma)$.  Suppose that
$F \in C^{1+\delta'}(\reals^d , L(\CA,\reals^d))$, 
$H \in C^{\frac{1}{\gamma}+\delta}(\reals^d , L(\CB,\reals^d))$, where $\delta,\delta'>0$,.
Then there exists $\beta_*=\beta_*(\gamma,\gamma',\delta') \in (\frac12,1)$ such that 
the solution map
$\Phi : \CC^{\beta,\gamma} \to C^{\gamma'}([0,T],\reals^d)$ given by
\[
\Phi(V,W,\BBW)=X
\]
is continuous for $\beta \in (\beta_*,1)$,
\end{thm}

The solution map $\Phi$ is a genuine extension of the classical solution map in the sense that, if $V,W$ are smooth paths and $\BBW$ is the iterated integral above $W$ (as in \eqref{e:iterated}) then $X = \Phi (V,W,\BBW)$ agrees with the solution to the classical ODE $dX = F(X)dV + H(X)dW$ with the same initial condition. 
\par
%
\begin{proof}[Proof of Theorem \ref{thm:rde}]
This is (a slight modification of) a standard result in rough path theory. 
Indeed when $V = 0$, 
it follows from \cite[Theorem 8.5]{frizhairer13}. 
The extension to nontrivial $V$ is a simple exercise in controlled rough paths.
\par
To apply rough path theory in Banach spaces one typically assumes an embedding 
\begin{equ}
L(\CB , L(\CB, \reals)) \hookrightarrow L(\CB \otimes \CB ,\reals)\;.
\end{equ}
See for instance \cite[Section 1.5]{frizhairer13}. We do not assume such an embedding. However, since we only interested in results concerning RDEs (and not general controlled rough paths) it is sufficient to assume the tensor product norm used to construct $\CB \otimes \CB$ is admissible. In particular, the only elements of $L(\CB , L(\CB, \reals))$ required to satisfy the above embedding are of product form. That is, they are described by $(f,g)x = f(x) g$ for all $x \in \CB$, with $f \in L(\CB,\reals)$ and $g \in L(\CB,\reals)$. Specifically, they are described by $(f,g) = (X'_k(t),\del_k H(X(t)))$ where $(X, X')$ is the controlled rough path candidate for the solution to the RDE. But clearly we can always perform such an embedding, by the identification $(f,g) \sim f \otimes g$ and by admissibility we have that  $f\otimes g \in L(\CB \otimes \CB , \reals)$ as required. 
\end{proof}

In the remainder of the article we will use the following result which is an immediate consequence of Theorem \ref{thm:rde}.

\begin{cor}\label{cor:rde}
Suppose that $V_\eps$, $W_\eps$ are smooth paths, 
and that $\BBW_\eps$ is the iterated integral of $W_\eps$ (as in \eqref{e:iterated}).
Let $\gamma \in (\frac13,\frac12]$.
Suppose that
$F \in C^{1+}(\reals^d , L(\CA,\reals^d))$ and $H \in C^{\frac{1}{\gamma}+}(\reals^d , L(\CB,\reals^d))$, 
and that
$X_\eps$ solves the ODE
\begin{equ}\label{e:rde_ode}
dX_\eps = F(X_\eps )dV_\eps + H(X_\eps) dW_\eps \quad
X_\eps(0) = \xi\;. 
\end{equ}
\par
If $(V_\eps,W_\eps,\BBW_\eps)\to_w (V,W,\BBW)$ in the $\CC^{\beta,\gamma}$ topology for all $\beta\in(\frac12,1)$, then 
$X_\eps \to_w X$ in the supnorm topology, where $X$ solves the RDE
\begin{equ}\label{e:rde_phi}
dX = F(X) dV + H(X)d\BW  \quad , \quad X(0) = \xi\;,
\end{equ}
with $\BW = (W,\BBW)$.
\end{cor}


Next, we list some properties of solutions to RDEs. Since these properties are completely standard, no proof will be given.

\begin{prop} \label{prop-X'soln}
When $X$ solves the RDE \eqref{e:rde_phi} we can always take $X_k'(\cdot)=H^k(X(\cdot))$ in the definition of the $d\BW$ integral in~\eqref{e:dW_int},~\eqref{e:dW_intS}.
\end{prop}

\begin{prop} \label{prop:CR}
Assume the set up of Theorem~\ref{thm:rde} 
and suppose moreover that 
$\BW = (W,\BBW) \in \CC^\gamma_g$.  Then the classical chain rule 
\begin{equ}
\vphi(X(t)) = \vphi(X(s)) + \sum_{k=1}^d \int_s^t \del_k \vphi(X) F^k (X)dV + \int_s^t \del_k \vphi(X) H^k (X)d\BW\;,
\end{equ}
is valid for any smooth $\vphi : \reals^d \to \reals$. 
\end{prop}
This result is an immediate consequence of the fact that the integrals are limits of smooth integrals, for which the chain rule holds. (The result fails for general rough paths $\BW \in \CC^\gamma$.)

The last result is an extension of the standard Kolmogorov continuity criterion to (smooth) rough paths, taking values in $\reals$. A proof can be found in \cite[Corollary 4]{gubinelli04}. The one dimensional case turns out to be sufficient for our needs, even in the Banach space setting.

\begin{lemma}\label{lem:kolm}
Let $T>0$ and let $W_\eps,\Wtilde_\eps : [0,T]\to \reals $ be smooth paths.
Define  
$I_\eps(s,t) = \int_s^t W_\eps(s,r) d\Wtilde_\eps(r)$.
Let $p>1$ and $\gamma\in (0,\frac12 - \frac{1}{2p})$,
and suppose that $M$, $\tilde M$ are constants.

\begin{itemize}
\item[(a)] If 
$(\BE |W_\eps(s,t)|^{2p})^{1/(2p)} \leq  M|t-s|^{1/2}$ 
for all $\eps>0$, $s,t\in[0,T]$, then 
there is a constant $C$ depending only on $T$, $d$, $p$, $\gamma$ such that
\begin{equ} 
\bigg(\BE \bigg( \sup_{s,t \in [0,T]} \frac{|W_\eps(s,t)|}{|t-s|^{\gamma}}\bigg)^{2p}\bigg)^{1/(2p)} \leq CM 
\;,\quad\text{for all $\eps>0$}\;.
\end{equ}
\item[(b)]
If
\begin{equ}
(\BE |W_\eps(s,t)|^{2p})^{1/(2p)} \leq M|t-s|^{1/2} \textand (\BE |\Wtilde_\eps(s,t)|^{2p})^{1/(2p)} \leq \tilde M|t-s|^{1/2}
\end{equ}
and 
\begin{equ}
(\BE |I_\eps(s,t)|^{p})^{1/p} \leq  M\tilde M|t-s|
\end{equ}
for all $\eps>0$, $s,t\in[0,T]$, then 
there is a constant $C$ depending only on $T$, $d$, $p$, $\gamma$ such that
\begin{equ}
 \bigg(\BE \bigg( \sup_{s,t \in [0,T]} \frac{|I_\eps(s,t)|}{|t-s|^{2\gamma}}\bigg)^{p}\bigg)^{1/p}\leq CM\tilde M
\;,\quad\text{for all $\eps>0$}\;.
\end{equ}
\end{itemize}
\end{lemma}

\section{The localized convergence result}\label{s:convergence}

In this section, we state the \emph{localized} version of Theorem \ref{thm:abstract}. 
\par
\begin{thm}\label{thm:fastslow_besov}
Suppose that Assumptions \ref{ass:clt} and \ref{ass:moments} hold with some $p \in (3,\infty]$ and $\kappa>0$. 
 Suppose that 
 $a\in C^{1+,0}(\reals^d\times M,\R^d)$ and $b\in C^{\alpha,\kappa}_0(\reals^d\times M,\R^d)$ for some $\alpha~>~2~+~\frac{2}{p-1}~+~\frac{d}{p}$.
Moreover, suppose that $a,b$ have compact support in the sense that there exists $E >0 $ such that $a(x,y) = b(x,y) = 0$ for any $|x| > E$ and $y \in M$. 
%
%
%
Then the conclusions from Theorem~\ref{thm:abstract} hold. 
\end{thm}

The proof is split into several steps:
\begin{enumerate}
\item In 
the remainder of this section, we reformulate $\xeps$ into a rough path framework 
and show that $\xeps$ solves a ODE of the form \eqref{e:rde_ode}. 
\item In Section \ref{s:conv_to_RDE}, we use the theory from Section \ref{s:rough} to show that $x_\eps \to_w X$ where $X$ is defined by an RDE of the form \eqref{e:rde_phi}.
\item In Section \ref{s:rde_ito}, we show that the RDE in step~2 can be re-written as the desired It\^o SDE.  
\end{enumerate}
The proof of Theorem \ref{thm:fastslow_besov}, which is a simple combination of the above facts, can be found at the end of Section \ref{s:rde_ito}. 
Then in Section \ref{s:localization}, we show how Theorem~\ref{thm:abstract} follows from  Theorem~\ref{thm:fastslow_besov}.

\subsection{The rough path reformulation of the fast-slow system}\label{s:formulation}
We define $C^\theta(\R^d,\R^d)$ to be the vector space of continuous functions $u:\R^d\to\R^d$ with components $u^1,\dots,u^d\in C^\theta(\R^d)$.   This is a 
Banach space with norm 
$\norm{u}_{C^\theta}=\sum_{i=1}^d \norm{u^i}_{C^\theta}$.

Since $\alpha>2+ 2/(p-1)+d/p$, we can choose $\theta>2+2/(p-1)$ such that
$\alpha>\theta+d/p$.
For the reformulation described in step~1 above, we take $\CA$ and $\CB$ to be the Holder spaces
$\CA = C^{1+}(\R^d,\R^d)$, $\CB = C^\theta(\R^d,\R^d)$. 
For $\eps>0$, we define the smooth paths
\begin{equ}
V_\eps(t) =   \int_0^t a(\cdot,\yeps(r))dr  \quad \text{and} \quad W_\eps(t) = \eps^{-1} \int_0^t b(\cdot,\yeps(r))dr, \quad t\in[0,T]\;.
\end{equ}

\begin{prop} \label{prop:VW}
If $a$ and $b$ are as in Theorem~\ref{thm:fastslow_besov}, then $V_\eps$ and $W_\eps$
take values in $\CA$ and $\CB$ respectively for each $\eps>0$, $t\in[0,T]$.
\end{prop}

\begin{proof}
By definition,
$\norm{W_\eps(t)}_{\CB} =\sum_{i=1}^d \| \eps^{-1}\int_0^t b^i(\cdot,\yeps(r))dr \|_{C^\theta}$.  
But
\begin{equ}
\sup_x\Big| \int_0^t D_x^k b^i(\cdot,\yeps(r))dr\Big| \leq  t \sup_x \sup_{y\in \Omega} |D_x^k b^i(x,y)|
\end{equ}
and similarly
\begin{equ}
\sup_{x,x'}\frac{| \int_0^t D_x^k  b^i(x,\yeps(r))dr - \int_0^t D_x^k  b^i(x',\yeps(r))dr|}{|x-x'|^{\floor{\theta} - \theta}} \leq t \sup_{x,x'} \sup_{y\in \Omega} \frac{|  D_x^k  b^i(x,y) - D_x^k  b^i(x',y)|}{|x-x'|^{\floor{\theta} - \theta}}\;.
\end{equ}
By definition of the Holder norm, it follows that 
\begin{equ}
\norm{W_\eps(t)}_{\CB} \leq \eps^{-1} t  \sum_{i=1}^d \sup_{y\in \Omega}\norm{b^i(\cdot,y)}_{C^\theta}=\eps^{-1}t\norm{b}_{C^{\theta,0}}
\end{equ}
which is finite by the assumption on $b$.
Similarly,
$\norm{V_\eps(t)}_{\CA} \leq  t \norm{a}_{C^{1+,0}} <\infty$.
\end{proof}

%
%

For $x\in\R^d$, we define the multidimensional Dirac distribution operator 
$H:\R^d\to L(C^\theta(\R^d,\R^d),\R^d)$ by setting $H(x)(u)=u(x)$.  
It is easily shown that 
\begin{prop} \label{prop:dirac}
$H \in C^\theta (\reals^d , L(C^\theta(\R^d,\R^d),\reals^d))$ for 
all $\theta\ge0$. 
\end{prop}
\par
In this way, we obtain operators 
$F:\R^d\to L(\CA,\R^d)$ and $H:\R^d\to L(\CB,\R^d)$.
The following result states that the above definitions are sufficient to reformulate \eqref{e:fastslow} in the rough path framework of Corollary \ref{cor:rde}.

\begin{lemma}\label{lem:xeps_rde}
Suppose that $a$ and $b$ are as in Theorem~\ref{thm:fastslow_besov} and
define $F,H,V_\eps,W_\eps$ as above. 
Then the solution $\xeps$ to the ODE \eqref{e:fastslow} satisfies the ODE
\begin{equ}\label{e:inceps}
d\xeps = F(\xeps) d\Veps + H(\xeps) d\Weps \quad , \quad \xeps(0) = \xi\;.
\end{equ}
Moreover, $F \in C^{1+}(\reals^d ; L(\CA,\reals^d))$ and $H \in C^{\theta}(\reals^d ; L(\CB,\reals^d))$ where $\theta>2+\frac{2}{p-1}$. 
\end{lemma}

\begin{proof}
The regularity of $F$ and $H$ follows immediately from Proposition~\ref{prop:dirac} and the choice of $\CA,\CB$.
%

For each fixed $y$, the function $x\mapsto a(x,y)$ is by assumption in $\CA$
and so the operation $F(x)a(\cdot,y)=a(x,y)$ is well-defined. 
Hence for fixed $x,t$,
\begin{equ}
F(x) \frac{dV_\eps(t)}{dt} =  F(x) a(\cdot,\yeps(t)) = a(x,\yeps(t))\;.
\end{equ}
Similarly,
$H(x) \frac{dW_\eps(t)}{dt} = \eps^{-1} b(x,\yeps(t))$.
It follows that 
\begin{equ}
\frac{d\xeps(t)}{dt} =  a(\xeps(t),\yeps(t)) + \eps^{-1} b(\xeps(t),\yeps(t)) 
 =F(\xeps(t))\frac{dV_\eps(t)}{dt} +  H(\xeps(t))\frac{dW_\eps(t)}{dt} \;.
\end{equ}
In the incremental form, we have precisely \eqref{e:inceps}. 
\end{proof}


\subsection{Tensor product of Holder spaces}
\label{s:holder_tensor}

As preparation for the application of rough path theory in Section~\ref{s:conv_to_RDE},
we define the tensor product $\CB\otimes\CB$ for the H\"older space
$\CB=C^{\theta}(\R^d,\R^d)$.

First we consider the scalar situation.
Define $C^{\theta , \theta}(\R^d\times\R^d)$ to be the space of continuous functions $u:\R^d\times\R^d\to\R$ with bounded norm
\begin{equ}\label{e:theta_theta}
\norm{u}_{C^{\theta,\theta}}= \sum_{|k|\leq \floor{\theta}}  \sup_x \norm{D_x^k u(x,\cdot) }_{C^\theta} +  \sum_{|k| =  \floor{\theta}}  \sup_{x,x'} \frac{\norm{\delta_{x,x'} D_x^k u(x,\cdot) }_{C^\theta}}{|x-x'|^{\theta-\floor{\theta}}}  \;,
\end{equ}
with the shorthand $\delta_{x,x'} u(x,z) = u(x,z) - u(x',z)$,
where the second summation is omitted if $\theta$ is an integer. 
Expanding the inner norm, we obtain 
\begin{equs}
\norm{u}_{C^{\theta,\theta} } &=  \sum_{|k| \leq \floor{\theta}, |l| \leq \floor{\theta}} \sup_{x,z} |D_x^k D_z^l u(x,z)|   +  \sum_{|k| = \floor{\theta}, |l| \leq \floor{\theta}} \sup_{x,x',z} \frac{| \delta_{x,x'} D_x^k D_z^l u(x,z)|}{|x-x'|^{\theta-\floor{\theta}}}   \\ & +  \sum_{|k| \leq \floor{\theta}, |l| =  \floor{\theta}} \sup_{x,z,z'} \frac{|\delta_{z,z'} D_x^k D_z^l  u(x,z)|}{|z-z'|^{\theta-\floor{\theta}}}   +  \sum_{|k|, |l| = \floor{\theta}} \sup_{x,x',z,z'} \frac{|\delta_{x,x'}\delta_{z,z'}D_x^k D_z^l u(x,z)|}{(|x-x'||z-z'|)^{\theta - \floor{\theta}}}  \;.
\end{equs}
Here, we use the shorthand $\delta_{z,z'} u(x,z) = u(x,z) - u(x,z')$ and $\delta_{x,x'} \delta_{z,z'} u(x,z) = u(x,z) - u(x',z) - u(x,z') + u(x',z') $. It follows that we could equally define the norm in \eqref{e:theta_theta} with the roles of $x$ and $z$ reversed. 


%
Let 
$\iota : C^\theta(\R^d) \otimes_a C^\theta(\R^d) \hookrightarrow C^{\theta,\theta}(\R^d\times\R^d)$ denote the embedding
\[
\iota\bigg(\sum_n u_n\otimes v_n\bigg)(x,z)=\sum_n u_n(x)v_n(z).
\]
Define the tensor product norm $\norm{\cdot}_{C^\theta\otimes C^\theta}$ by setting 
$\norm{\sum_n u_n\otimes v_n}_{C^\theta \otimes C^\theta}=
\norm{\iota(\sum_n u_n~\otimes~v_n)}_{C^{\theta,\theta}}$, and take the completion
 to obtain the tensor product space $(C^\theta(\R^d)\otimes C^\theta(\R^d),\norm{\cdot}_{C^\theta\otimes C^\theta})$.

\begin{prop} \label{prop:tensor}
The tensor product norm $\norm{\cdot}_{C^\theta \otimes C^\theta}$ is admissible. 
\end{prop}

\begin{proof}
By an obvious factorization we have that $\norm{u\otimes v}_{C^\theta\otimes C^\theta} = \norm{u}_{C^\theta} \norm{v}_{C^\theta}$ for every $u,v \in C^\theta(\R^d)$. It remains to check that $\norm{f\otimes g}_{L(C^\theta \otimes_a C^\theta,\reals)} \leq \norm{f}_{L(C^\theta,\reals)} \norm{g}_{L(C^\theta,\reals)}$ for all $f,g \in L(C^\theta,\reals)$. Notice that
\begin{equs}\label{e:fgbound}
\big| (f\otimes g) \sum_{n} u_n \otimes v_n \big| &= \big| \sum_{n} f(u_n) g(v_n)  \big| = \big| f\big( \sum_{n} u_n g(v_n) \big) \big| \\ & \leq \norm{f}_{L(C^\theta,\reals)}  \big\|\sum_{n} u_n g(v_n)\big\|_{C^\theta}\;.
\end{equs}
But 
\begin{equs}
\big\|\sum_{n} u_n g(v_n)\big\|_{C^\theta} &= \sum_{|k|\leq \floor{\theta}}\sup_x |\sum_n D_x^k u_n(x) g(v_n)| +  \sum_{|k|= \floor{\theta}}  \sup_{x,x'}\frac{| \sum_{n} \delta_{x,x'}D_x^k  u_n(x) g(v_n)  |}{|x-x'|^{\theta - \floor{\theta}}} \;.
\end{equs}
For each fixed $x$,
\begin{equs}
\big|\sum_n D^k_x u_n(x) g(v_n)\big|  & = \big|g\big( \sum_n D_x^k u_n(x) v_n\big)\big| 
\leq \norm{g}_{L(C^\theta,\reals)} \big\|\sum_{n} D_x^k u_n(x) v_n\|_{C^\theta} \\ &
\leq \norm{g}_{L(C^\theta,\reals)} \sum_{n} |D_x^k u_n(x)|\norm{v_n}_{C^\theta} \;,
\end{equs}
and similarly $| \sum_{n} \delta_{x,x'} D_x^k u_n(x) g(v_n)  |\le \norm{g}_{L(C^\theta,\reals)}\sum_{n} |\delta_{x,x'}D_x^k u_n(x)|\norm{ v_n}_{C^\theta}$. Substituting this back into \eqref{e:fgbound}, we obtain
\begin{equ}
\big| (f\otimes g) \sum_{n} u_n \otimes v_n  \big| \leq \norm{f}_{L(C^\theta,\reals)} \norm{g}_{L(C^\theta,\reals)} \big\|\iota\big( \sum_{n} u_n \otimes v_n \big) \|_{C^{\theta,\theta}}\;.
\end{equ}
Hence $\norm{f\otimes g}_{L(C^\theta \otimes C^\theta,\reals)} \leq \norm{f}_{L(C^\theta,\reals)} \norm{g}_{L(C^\theta,\reals)}$. 
\end{proof}
%

%
Next, we define the tensor product $\CB\otimes\CB=C^{\theta}(\R^d,\R^d)\otimes C^{\theta}(\R^d,\R^d)$ to be the space
of \mbox{$d\times d$} ``matrices'' with entries in $C^\theta(\R^d) \otimes C^\theta(\R^d)$, endowed with
the norm \newline $\norm{\sum_n u_n\otimes v_n}_{C^\theta\otimes C^\theta}=
\sum_{i,j=1}^d \norm{\sum_n u_n^i\otimes v_n^j}_{C^\theta\otimes C^\theta}$.

\begin{cor} \label{cor:tensor}
$C^\theta(\R^d,\R^d)\otimes C^\theta(\R^d,\R^d)$ is a Banach space with admissible tensor product norm $\norm{\cdot}_{C^\theta\otimes C^\theta}$.
\end{cor}

\begin{proof}
Completeness is an immediate consequence of the completeness of $C^\theta(\R^d)\otimes C^\theta(\R^d)$.   
Admissibility of $\norm{\cdot}_{C^\theta\otimes C^\theta}$ is proved by a calculation similar to the one in Proposition~\ref{prop:tensor}.
\end{proof}

\section{Convergence to the RDE}\label{s:conv_to_RDE}

The objective of this section is to use Corollary \ref{cor:rde} to characterize the $\eps \to 0$ limit of the  solution $x_\eps$ for the
fast-slow ODE~\eqref{e:fastslow} as the solution to an RDE. 

We suppose throughout that Assumptions~\ref{ass:clt} and~\ref{ass:moments}
are valid with $p>3$ and $\kappa>0$, and that
$a$, $b$ are as in Theorem~\ref{thm:fastslow_besov}.
Define $V_\eps:[0,T]\to\CA$ and $W_\eps:[0,T]\to\CB$
as in Section~\ref{s:formulation};  in particular, $\CA=C^{1+}(\R^d,\R^d)$
and $\CB = C^\theta(\R^d,\R^d)$
where $\theta>2+2/(p-1)$ and $\alpha>\theta+d/p$.
Define $\CB \otimes \CB$ as in Section \ref{s:holder_tensor}, and 
the iterated integral $\BBW_\eps : [0,T] \to \CB \otimes \CB$ by
\begin{equ}
\BBW_\eps(t) = \int_0^t W_\eps \otimes dW_\eps  = \eps^{-2} \int_0^t \int_0^r b(\cdot,\yeps(u))\otimes b(\cdot,\yeps(r)) du dr\;,
\end{equ}
The integral is well defined by Remark \ref{rmk:tensor1}. 
Let $\BW_\eps=(W_\eps,\BBW_\eps):[0,T]\to \CB\times(\CB\otimes\CB)$. Define $\abar\in\CA$ by
\begin{equ}\label{e:Veq}
\abar =  \int_\Omega a(\cdot, y) d\mu(y) \;.
\end{equ}
We now state the main result of this section.
\begin{thm}\label{thm:xeps_limit}
The family $\{\xeps\}_{\eps>0}$ is tight in $C([0,T] ; \reals^d)$. Moreover, every limit point $X$ satisfies an RDE of the form
\begin{equ}\label{e:lemma_RDE}
dX = F(X) \abar dt + H(X) d\BW \quad , \quad X(0) = \xi\;,
\end{equ}
where $\BW$ is a limit point of $\{\BW_\eps\}_{\eps > 0}$
in $\CC^\gamma$ for all $\gamma\in(\frac13,\frac12-\frac{1}{2p})$. 
\end{thm}

\begin{rmk} \label{rmk:xeps_limit}
Evidently it suffices to prove Theorem~\ref{thm:xeps_limit}
for $\gamma$ arbitrarily close to $\gamma^*=\frac12-\frac{1}{2p}$.
By Lemma~\ref{lem:xeps_rde}, $H$ is $C^{\theta}$ where
$\theta>2+\frac{2}{p-1}=\frac{1}{\gamma^*}$.
Hence $\theta>\frac{1}{\gamma}$ for $\gamma$ close to $\gamma^*$ ensuring that
$H$ has the regularity required in Corollary~\ref{cor:rde}.
\end{rmk}

The second aim of this section is to characterise the finite-dimensional distributions of the limit points of $\BW_\eps$.   This is done 
in Lemma~\ref{lem:fdd}. 

To control the tightness of $\BW_\eps$, we make use of Besov spaces described in
Subsection~\ref{s:besov}.
In Subsection~\ref{s:tight}, we prove tightness of $(V_\eps,\BW_\eps)$ and
deduce tightness of $x_\eps$.
In Subsection~\ref{s:limit}, we complete the proof of Theorem \ref{thm:xeps_limit} and characterize the limit points of 
$(V_\eps,\BW_\eps)$. 


\subsection{Besov spaces}
\label{s:besov}
Let $s>0$ and fix (arbitrarily) an integer $m>s$. The classical Besov space 
$B_p^s=B_{p}^s(\reals^d)$ can be defined (for all $p\geq 1$) as the set of all $L_p$ functions $u:\R^d\to\R$ such that 
\begin{equ}
\norm{u}_{B_p^s} =  \bigg(\norm{u}_{L_p}^p + \int_{|\sigma|\leq 1} |\sigma|^{-sp-d} \norm{\Delta_\sigma^m u}_{L_p}^p  d\sigma \bigg)^{1/p} < \infty
\end{equ}  
where
\begin{equ}
\Delta_\sigma u (x) = u(x+\sigma) - u(x) \textand \Delta_\sigma^{l+1} = \Delta_\sigma \circ \Delta_\sigma^l
\end{equ}
and $\norm{\cdot}_{L_p}$ is the standard $L_p$ norm on $\reals^d$. 
For more details, see \cite{tri1,tri3}. 
\begin{rmk}
The classical Besov spaces $B_{p,q}^s$ typically come with two indices of integrability and norm $\norm{u}_{B_{p,q}^s} =  \bigg( \norm{u}_{L_p}^{q} + \int_{|\sigma|\leq 1} |\sigma|^{-sq-d} \norm{\Delta_\sigma^m u}_{L_p}^q  d\sigma \bigg)^{1/q}$. In this article, we always take $p=q$. Hence our Besov spaces $B_p^s$ are really the same as the Slobodeckij spaces, when $s \neq \naturals$. The norm we employ is not the most standard choice but is well known to be equivalent to the usual Besov norm \cite[Section 2.5.12]{tri1}. 
\end{rmk}

For $\kappa\in[0,1)$, we also introduce a norm on functions $u=u(x,y)$ that are $B_p^s$ in the $x$ variable and $C^\kappa$ in the $y$ variable: 
\begin{equ}
\norm{u}_{B_p^s ; C^\kappa} =  \bigg( \int \norm{u(x,\cdot)}_{C^\kappa}^p dx + \int_{|\sigma|\leq 1} \sigma^{-sp-d} \bigg( \int \norm{\Delta^m_\sigma u(x,\cdot)}_{C^\kappa}^p dx \bigg)d\sigma\bigg)^{1/p}\;,
\end{equ}
with $\Delta_{\sigma}^m$ acting only in the $x$ component. 

\begin{lemma}\label{lem:besov_embed}
We have the embeddings
\begin{equ}
\norm{u}_{C^\theta} \lesssim \norm{u}_{B_p^{\theta + d/p}}
\end{equ}
and 
\begin{equ}
\norm{u}_{C^{\theta,\theta}} \lesssim \bigg( \int \norm{u(x,\cdot)}^p_{B_p^{\theta + d/p}} dx + \int_{|\sigma|\leq 1} |\sigma|^{-p\theta - 2d} \int \norm{\Delta^m_{x,\sigma} u(x,\cdot)}^p_{B_p^{\theta + d/p}}dx d\sigma \bigg)^{1/p}\;.
\end{equ}

\end{lemma}

\begin{proof}
The first estimate can be found in \cite[Section 2.7.1]{tri1}. The second estimate is obtained by applying the first estimate in the $x$ coordinate (for each fixed $z$) and then in the $z$ coordinate (for each fixed $x$). 
\end{proof}

\begin{lemma}\label{lem:compact}
If $u \in C^{\alpha,\kappa}(\R^d\times M)$ and has compact support (in the sense of Theorem \ref{thm:fastslow_besov}) then $\norm{u}_{B_p^s ; C^\kappa} < \infty$ for any $s < \alpha$. 
\end{lemma}

\begin{proof}
Since $m>s$ is arbitrary, it suffices to take $m=\lceil s\rceil$.
The $\norm{u}_{L_p}$ part is obviously finite for any $p\geq 1$, since $u$ is bounded and has compact support. Hence it suffices to bound the semi-norm part of the Besov norm. We claim that 
\begin{equ}\label{e:Deltam_ineq}
\sup_{x\in\reals^d}\norm{\Delta^m_{\sigma} u(x,\cdot)}_{C^\kappa} \leq \norm{u}_{C^{\alpha,\kappa}} |\sigma|^\alpha\;.
\end{equ}
In this case
\begin{equ}
\int_{|\sigma|\leq 1} |\sigma|^{-sp-d} \int \norm{\Delta^m_{x,\sigma} u(x,\cdot)}_{C^\kappa}^p dx d\sigma  \leq \norm{u}_{C^{\alpha,\kappa}}^p  \int_{|\sigma|\leq 1} |\sigma|^{-sp-d} |\sigma|^{\alpha p} d\sigma  <\infty
\end{equ}
as required, since $\alpha > s$. 

All that is left is to prove the inequality \eqref{e:Deltam_ineq}. By the chain rule, we have that
\begin{equ}
u(x + \sigma,y) - u(x,y) = \int_0^1 D_x u(x + (1-\zeta)\sigma ,y) d\zeta \cdot \sigma \;.
\end{equ}
Repeating this identity, and writing $\delta_{y,y'} u(x,y) = u(x,y) - u(x,y')$ we obtain
\begin{equs}
\Delta^{m-1}_\sigma &u (x,y) - \Delta^{m-1}_\sigma u (x,y')\\ &=  \int_{[0,1]^{m-1}} D_x^{m-1} \delta_{y,y'} u(x + ((m-1)-(\zeta_1 + \dots + \zeta_{m-1})\sigma , y )d\zeta \cdot \sigma^{m-1}\;.
\end{equs}
It follows easily that 
\begin{equs}
|\Delta_\sigma^m u(x,y) - \Delta_\sigma^m u(x,y')|  &\leq \sup_x |D^{m-1} \delta_{y,y'} u (x+\sigma,y) - D^{m-1} \delta_{y,y'} u (x,y) | |\sigma|^{m-1}\\ &\leq \norm{u}_{C^{m-1+\zeta,\kappa}} |\sigma|^{m-1 + \zeta} |y-y'|^\kappa\;. 
\end{equs}
This proves \eqref{e:Deltam_ineq}.
\end{proof}

\subsection{Tightness of $(V_\eps,\BW_\eps)_{\eps>0}$ and $(x_\eps)_{\eps>0}$}
\label{s:tight}

Firstly, we estimate 
$\norm{V_\eps(s,t)}_{\CA}$.

\begin{lemma} \label{lem:tightV}
We have that 
$\sup_{\eps > 0}\norm{V_\eps(s,t)}_{\CA} \lesssim \norm{a}_{C^{\eta,0}} |t-s|$
uniformly over $\Omega$. 
\end{lemma}

\begin{proof}
Without loss, we suppose that $a$ is real-valued.
Write $V_\eps(s,t;x) = \int_s^t a(x,\yeps(r))dr$.  We have
\begin{equ}
|D_x^k V_\eps(s,t;x)| \leq \norm{D_x^k a(x,\cdot)}_{C^0} |t-s| ,\quad
|\delta_{x,x'}D_x^k  V_\eps(s,t;x)| \leq \norm{\delta_{x,x'}D_x^k a^i(x,\cdot)}_{C^0} |t-s| 
\end{equ}
and hence
\begin{equs} 
& \norm{V_\eps(s,t)}_{\CA} 
  = \sum_{|k|\leq \floor{\eta}} \sup_x |D_x^k V_\eps(s,t;x)| 
  +  \sum_{|k| = \floor{\eta}} \sup_{x,x'} \frac{|\delta_{x,x'}D_x^k V_\eps(s,t; x)|}{|x-x'|^{\eta - \floor{\eta}}} \\
& \qquad  \le \bigg( \sum_{|k|\leq \floor{\eta}} \sup_x \norm{D^k_x a(x,\cdot)}_{C^0} 
  +  \sum_{|k| = \floor{\eta}} \sup_{x,x'} \frac{\norm{\delta_{x,x'}D_x^k a(x,\cdot)}_{C^0}}{|x-x'|^{\eta - \floor{\eta}}} \bigg) |t-s| \\
&\qquad = \norm{a}_{C^{\eta,0}} |t-s|\;,
\end{equs} 
as required.
\end{proof}

We now obtain an analogous estimate for $\BW_\eps$. Again, we may suppose without loss that $b$ is real-valued.  First, we introduce the notation
\[
\Delta_\sigma^m W_\eps(s,t;x)=\eps^{-1}\int_s^t\Delta_\sigma^mb(x,y_\eps(r))dr\;,
\]
for $m\ge0$, where the operator $\Delta_\sigma^m$ is omitted when $m=0$.
Similarly, we write
\[
\Delta_\sigma^m\Delta_{\sigma'}^{m'}\BBW_\eps(s,t;x,x')
=\eps^{-2}\int_s^t\int_s^r
\Delta_\sigma^mb(x,y_\eps(u))
\Delta_{\sigma'}^{m'}b(x',y_\eps(r)) dudr\;.
\]

\begin{prop} \label{prop:tight}
\begin{itemize}
\item[(a)]
$\BE_\mu \bigg(\displaystyle\sup_{s,t} \frac{|\Delta_\sigma^m  W_\eps(s,t; x)|}{|t-s|^\gamma}\bigg)^{2p} \lesssim \norm{\Delta_\sigma^m b(x,\cdot)}_{C^\kappa}^{2p} $.
\item[(b)]
$\BE_\mu \bigg(\displaystyle\sup_{s,t} \frac{|\Delta_{x,\sigma}^m \Delta_{x',\sigma'}^{m'} \BBW_\eps(s,t; x , x')|}{|t-s|^{2\gamma}}\bigg)^{p} \lesssim \norm{\Delta_{x,\sigma}^m b(x,\cdot)}_{C^\kappa}^p \norm{\Delta_{x',\sigma'}^{m'} b(x',\cdot)}_{C^\kappa}^p $.
\end{itemize}
\end{prop}

\begin{proof}
Recall from the introduction that $y_\eps(t)=y(t\eps^{-2})=\phi_{t\eps^{-2}}(y_0)$ where $\phi$ is the underlying fast flow
and $y_0\in\Omega$ is the initial condition.   Hence by change of variables,
\begin{equ}
\Delta_\sigma^m  W_\eps(s,t; x)(y)
=\eps\int_{s\eps^{-2}}^{t\eps^{-2}}\Delta_\sigma^mb(x,\phi_r y)dr.
\end{equ}
But $\Delta_\sigma^mb(x,\cdot)\in C^\kappa_0(\Omega,\reals)$ for each $x$, $\sigma$, so  by Proposition \ref{prop:moments},
\[
(\BE_\mu |\Delta_\sigma^mW_\eps(s,t;x)|^{2p})^{1/(2p)} 
\lesssim \norm{\Delta_\sigma^mb(x,\cdot)}_{C^\kappa} |t-s|^{1/2},
\]
uniformly in $s,t,x,\sigma,\eps$.
Hence part (a) follows from the Kolmogorov criterion, Lemma~\ref{lem:kolm}(a).
Part (b) is proved almost identically using Lemma \ref{lem:kolm}(b).
\end{proof}

\begin{lemma} \label{lem:tight}
We have that $\sup_{\eps > 0}\BE_\mu \normy{\BW_\eps}_{\CC^\gamma}^{2p} < \infty$ 
for any $\gamma \in(\frac13, \frac12 - \frac{1}{2p})$.
\end{lemma}

\begin{proof}
%
%
%
%
%

By Lemma \ref{lem:besov_embed}, $\norm{W_\eps(s,t)}_{\CB}\lesssim\norm{W_\eps(s,t)}_{B_p^{\theta+d/p}}$ and hence
\begin{equs} 
& \sup_{s,t} \frac{\norm{W_\eps(s,t)}_{\CB}}{|t-s|^\gamma} \\
&\lesssim \sup_{s,t} \frac{1}{|s-t|^\gamma}\bigg( \int  |W_\eps(s,t ; x)|^p  dx 
 + \int_{|\sigma|\leq 1} |\sigma|^{-\theta p-2 d} \int |\Delta_\sigma ^m W_\eps(s,t; x)|^p dx d\sigma \bigg)^{1/p} \;.
\end{equs}
Taking the supremum inside the integrals and using the inequality $(x+y)^{1/p} \leq x^{1/p} + y^{1/p}$, 
\begin{equs} \label{e:firstmo}
& \sup_{s,t} \frac{\norm{W_\eps(s,t)}_{\CB}}{|t-s|^\gamma} 
\leq \bigg( \int \bigg(  \sup_{s,t} \frac{|W_\eps(s,t ; x)|}{|t-s|^{\gamma}}\bigg)^p dx \bigg)^{1/p}\\
& \qquad + \bigg( \int_{|\sigma|\leq 1} |\sigma|^{-\theta p-2d} \int \bigg(\sup_{s,t} \frac{|\Delta_\sigma ^m W_\eps(s,t; x)|}{|t-s|^{\gamma}}\bigg)^p dx d\sigma \bigg)^{1/p} = 
\sum_{k=1}^2 \bigg(\int c_k^p dz\bigg)^{1/p}\;,
\end{equs}
where $dz=dx$ or $dz=|\sigma|^{-\theta p-2d}dxd\sigma$ respectively.
Applying the triangle inequality, first for $L_{2p}$ and then for $L_2$,
\begin{equs} 
& \norm{\sum_k   \bigg(\int c_k^p dz\bigg)^{1/p}}_{L_{2p}(d\mu)}  \le
\sum_k \norm{\bigg(\int c_k^p dz\bigg)^{1/p}}_{L_{2p}(d\mu)} =
\sum_k\bigg(\;\BE_\mu\bigg(\int c_k^p dz\biggr)^2\;\biggr)^{1/(2p)}
\\ & \quad  =\sum_k\bigg(\norm{\int c_k^p}_{L_2(d\mu)}\;\biggr)^{1/p}
 \le \sum_k\bigg(\int \norm{c_k^p}_{L_2(d\mu)}dz\;\biggr)^{1/p}
 = \sum_k\bigg(\int\bigg( \BE_\mu c_k^{2p}\biggr)^{1/2}dz\biggr)^{1/p}\;.
\end{equs}
Substituting into~\eqref{e:firstmo} and applying Proposition~\ref{prop:tight}
to each term, we obtain
\begin{equs} \label{eq-W}
& \bigg(\BE_\mu\bigg(\sup_{s,t} \frac{\norm{W_\eps(s,t)}_{\CB}}{|t-s|^\gamma} \bigg)^{2p}\bigg)^{1/(2p)} 
\le  \bigg( \int\bigg(\BE_\mu \bigg(\sup_{s,t} \frac{| W_\eps(s,t; x)|}{|t-s|^{\gamma}}\bigg)^{2p}\Bigg)^{1/2} dz \bigg)^{1/p}  \\
& \qquad \qquad+ \bigg( \int\bigg(\BE_\mu \bigg(\sup_{s,t} \frac{|\Delta_\sigma ^m W_\eps(s,t; x)|}{|t-s|^{\gamma}}\bigg)^{2p}\Bigg)^{1/2} dz \bigg)^{1/p}  \\
& \qquad \lesssim
\bigg( \int \norm{b(x,\cdot)}_{C^\kappa}^p dz  \bigg)^{1/p} 
  + \bigg( \int \norm{\Delta^m_\sigma b(x,\cdot)}_{C^\kappa}^p dz \bigg)^{1/p} 
 \lesssim \norm{b}_{B_p^{\theta + d/p} ; C^\kappa}<\infty \;,
\end{equs}
where the last inequality follows from
Lemma \ref{lem:compact} (since $\theta+d/p<\alpha$.)
\par
We now use the same method to estimate the $\BBW_\eps$ term. Just as above, via Lemma \ref{lem:besov_embed} we have
\begin{equs}
& \sup_{s,t} \frac{\norm{\BBW_\eps(s,t)}_{\CB\otimes\CB}}{|t-s|^{2\gamma}}
=\sup_{s,t} \norm{\BBW_\eps(s,t ; \cdot,\cdot)}_{C^{\theta,\theta}}
\leq \bigg( \int  \bigg(\sup_{s,t} \frac{|\BBW_\eps(s,t; x , x')|}{|t-s|^{2\gamma}}\bigg)^p dz \bigg)^{1/p} \\ 
&+\bigg(\int  \bigg(\sup_{s,t} \frac{|\Delta^m_{x,\sigma} \BBW_\eps(s,t; x , x')|}{|t-s|^{2\gamma}}\bigg)^p dz  \bigg)^{1/p} 
+\bigg(\int  \bigg(\sup_{s,t} \frac{|\Delta^m_{x',\sigma'} \BBW_\eps(s,t; x , x')|}{|t-s|^{2\gamma}}\bigg)^p dz \bigg)^{1/p} \\ 
&+\bigg(\int  \bigg(\sup_{s,t} \frac{|\Delta^m_{x,\sigma}\Delta^m_{x',\sigma'} \BBW_\eps(s,t; x , x')|}{|t-s|^{2\gamma}}\bigg)^p dz \bigg)^{1/p}\;,
\end{equs}
where $dz$ is variously $dxdx'$, 
$|\sigma|^{-\theta p-2d}dxdx'd\sigma$,
$|\sigma'|^{-\theta p-2d}dxdx'd\sigma'$ or
$|\sigma\sigma'|^{-\theta p-2d}dxdx'd\sigma d\sigma'$.
We apply $\BE_\mu$ to the left hand side, using the triangle inequality to
take the $L_{1}$ norm inside the sums and integrals.
Applying Proposition~\ref{prop:tight}(b) to each term, we obtain
\begin{equs} \label{eq-WW}
  \BE_\mu & \sup_{s,t} \frac {\norm{\BBW_\eps(s,t)}_{\CB\otimes\CB}}{|t-s|^{2\gamma}}
  \lesssim \bigg(\int \norm{b(x,\cdot)}_{C^\kappa}^p \norm{b(x',\cdot)}_{C^\kappa}^p dz \bigg)^{1/p}
 \\ & +\bigg(\int  \norm{\Delta_{x,\sigma}^mb(x,\cdot)}_{C^\kappa}^p \norm{b(x',\cdot)}_{C^\kappa}^p dz \bigg)^{1/p}
  +\bigg(\int  \norm{b(x,\cdot)}_{C^\kappa}^p \norm{\Delta_{x',\sigma'}^{m'}b(x',\cdot)}_{C^\kappa}^p dz \bigg)^{1/p}
 \\ & +\bigg(\int \norm{\Delta_{x,\sigma}^mb(x,\cdot)}_{C^\kappa}^p \norm{\Delta_{x',\sigma'}^{m'}b(x',\cdot)}_{C^\kappa}^p dz \bigg)^{1/p}\;.
%
%
%
\\&=\bigg\{ \bigg(\int \norm{b(x,\cdot)}_{C^\kappa}^p dx \bigg)^{1/p} + \bigg( \int_{|\sigma|\leq 1} |\sigma|^{-\theta p-2d} \int \norm{\Delta^m_{x,\sigma} b (x,\cdot) }_{C^\kappa}^p dx d\sigma \bigg)^{1/p}  \bigg\} \\
& \times \bigg\{ \bigg(\int \norm{b(x',\cdot)}_{C^\kappa}^p dx' \bigg)^{1/p} + \bigg( \int_{|\sigma'|\leq 1} |\sigma'|^{-\theta p-2d} \int \norm{\Delta^m_{x',\sigma} b (z,\cdot) }_{C^\kappa}^p dx' d\sigma'  \bigg)^{1/p}  \bigg\} \\
&\lesssim \norm{b}_{B_p^{\theta + d/p} ; C^\kappa}^2<\infty\;,
\end{equs}
where the last inequality follows from Lemma \ref{lem:compact}. Combining~\eqref{eq-W} and~\eqref{eq-WW}, we obtain the required estimate 
for $\BW_\eps$.
\end{proof}

Finally, we have the claimed tightness result. 

\begin{cor} \label{cor:tight}
\begin{itemize}
\item[(a)] The family $(V_\eps,\BW_\eps)_{\eps>0}$ is tight
in $\CC^{\beta,\gamma}$
for any $\beta\in(\frac12,1)$, $\gamma \in(\frac13, \frac12~-~\frac{1}{2p})$.
\item[(b)] The family $(x_\eps)_{\eps>0}$ is tight
in $C([0,T],\R^d)$.
\end{itemize}
\end{cor}

\begin{proof}
We first show that $\BW_\eps = (W_\eps,\BBW_\eps)$ is tight in $\CC^\gamma$. 
Let $R>2$, $\gamma'\in(\gamma,\frac12-\frac{1}{2p})$, and let $B_R\subset \CC^\gamma$ be the ball of radius $R$ in the $\rho_{\gamma'}$ metric. 
By a standard Arzela-Ascoli argument (for instance, see \cite[Chapter 5]{friz10}) one can show that $B_R$ is sequentially compact with respect to $\rho_\gamma$ and hence compact in $\CC^\gamma$. 
Since $\rho_{\gamma'} (\BW_\eps,0) \leq \normy{\BW_\eps}_{\CC^{\gamma'}} + \normy{\BW_\eps}^{2}_{\CC^{\gamma'}} $ and $R>2$, 
\begin{equ}
\BPmnu \big( \BW_\eps \not\in B_R \big) \leq \BPmnu \big( \normy{\BW_\eps}_{\CC^{\gamma'}} \geq  (R/2)^{1/2} \big)\;.
\end{equ}
Hence by Markov's inequality  and Lemma~\ref{lem:tight},
\begin{equ}
\BPmnu \big( \BW_\eps \not\in B_R \big) \leq 
2^p \BE_\mu \normy{\BW_\eps}_{\CC^{\gamma'}}^{2p} / R^p
\lesssim
R^{-p}\;.
\end{equ}
This proves tightness of $\BW_\eps$.
An analogous, but simpler, argument using Lemma~\ref{lem:tightV}
shows that
$V_\eps$ is tight in $\CC^\beta$, concluding the proof of part (a). 

For part (b),
let $x_{\eps_k}$ be a subsequence.
By part (a), we can apply Prokhorov's theorem to $(V_{\eps_k},\BW_{\eps_k})$.  Hence
passing to a subsubsequence,
there exists $(V,\BW)\in \CC^{\beta,\gamma}$ such that $(V_{\eps_k},\BW_{\eps_k})\to_w (V,\BW)$ in the $\CC^{\beta,\gamma}$ topology. 
By Lemma~\ref{lem:xeps_rde} and
Corollary \ref{cor:rde}, $x_{\eps_k}\to_w  X$ in 
$C([0,T],\R^d)$ where $X$ satisfies the 
RDE~\eqref{e:inceps} driven by $(V,\BW)$.  It follows that 
$\{x_\eps\}_{\eps>0}$ is weakly precompact in $C([0,T],\R^d)$.
Since $C([0,T],\R^d)$ is Polish, we can apply Prokhorov's theorem to deduce
 that $\{x_\eps\}_{\eps > 0}$ is tight.
\end{proof}

\subsection{Characterization of limits of $(V_\eps,\BW_\eps)_{\eps>0}$ and $(x_\eps)_{\eps>0}$}
\label{s:limit}

We begin by describing the limit of $V_\eps$.

\begin{lemma} \label{lem:V}  
Define the deterministic element $V\in C^1([0,T],\CA)$ given by
$V(t) = \abar t$ where $\abar \in \CA$ is defined in~\eqref{e:Veq}.
Then
$V_\eps\to V$ in probability in $\CC^\beta$ 
for any $\beta\in(\frac12,1)$, 
\end{lemma}

\begin{proof}
Let $\pi\in L(\CB,\R)$.
Then $\pi V_\eps(t)=\eps^2\int_0^{t\eps^{-2}}(\pi a)\circ \phi_s\,ds$.
By ergodicity of $\mu$, it follows from the ergodic theorem that
$\pi V_\eps(1)\to\pi V(1)$ almost surely. 
By Lemma~\ref{lem-elem}, 
$\pi V_\eps\to\pi V$ almost surely, and hence in probability, in 
$C([0,T],\R)$.

Suppose for contradiction that $V_\eps$ fails to converge weakly to $V$ in $\CC^\beta$.
By Corollary~\ref{cor:tight}, the family $V_\eps$ is tight in $\CC^\beta$, so
there is a subsequence such that
$V_{\eps_k}\to_w Z$ in $\CC^\beta$ where the random 
process $Z$ differs from $V$.  In particular, 
$\pi V_{\eps_k}(t_0)\to_w \pi Z(t_0)$ in $\reals$ for any $\pi\in L(\CA,\reals)$ and any $t_0\in[0,T]$.
Hence $\pi Z(t_0)$ has the same distribution as 
$\pi V(t_0)$ and so $\BP\big(\pi Z(t_0)=\pi\bar a t_0\big)=1$.
Since $\pi$ is arbitrary, it follows that 
$\BP\big(Z(t_0)=\bar a t_0\big)=1$.
But $Z$ is continuous, so $Z=V$ with probability one, giving the desired 
contradiction.
\end{proof}

\begin{proof}[Proof of Theorem~\ref{thm:xeps_limit}]
We have shown in Corollary \ref{cor:tight} that $\{\xeps\}_{\eps>0}$ is tight.  Let $X$ be a limit point, with $x_{\eps_k}\to_w X$ in $C([0,T] ; \reals^d)$.
By Lemma~\ref{lem:tight}, we can pass to a subsubsequence for which
$(V_{\eps_k},\BW_{\eps_k})$ converges weakly in $\CC^{\beta,\gamma}$.
Denote the limit by $(V,\BW)$.  By Lemma~\ref{lem:xeps_rde} 
and Corollary~\ref{cor:rde}, $X$ solves an RDE of the form~\eqref{e:inceps}
driven by $(V,\BW)$.
By Lemma~\ref{lem:V}, $V(t)=\bar at$ completing the proof.
\end{proof}

Finally, we obtain a partial (see Remark~\ref{rmk-partial}) characterization 
of the limit points of $\BW_\eps$ in terms of 
 their finite dimensional distributions. 
For each fixed $\pi \in L( \CB,\reals^m)$,
let $(\BP^\pi,\Omega^\pi, \CF^\pi)$ be a probability space endowed with a filtration $\{\CF_t^\pi\}_{t\geq0}$ rich enough to support Brownian motion. We define a stochastic process $(B_\pi,\BBB_\pi) : [0,T] \to \reals^{m}\times\reals^{m\times m}$ on the probability space $(\BP^\pi,\Omega^\pi, \CF^\pi)$, where $B_\pi$ is a $\reals^m$-valued $\CF^\pi_t$ - Brownian motion with covariance
\begin{equ}\label{e:Bpi_cov}
\BE^\pi B_\pi^i(1)B_\pi^j(1) = \auto(\pi^i b , \pi^j b ) + \auto(\pi^j b , \pi^i b ) 
\end{equ}
and $\BBB_\pi$ is defined by
\begin{equ}\label{e:Bpi_ito}
\BBB_\pi^{ij}(t) = \int_0^t B_\pi^i  dB_\pi^j + \auto(\pi^i b , \pi^j b) t
\end{equ}
where the integral is of It\^o type. Notice that this is precisely the structure that arises under Assumption \ref{ass:clt}.

\begin{rmk}
Here $\pi^i b$ denotes the observable $y \mapsto \pi^i b(\cdot,y)$, with $\pi^i$ acting on $b$ as a function of $x$. By the regularity assumptions on $b$, it is easy to check that $\pi^i b\in C^\kappa_0(\Omega,\reals)$ and lies in the domain of $\auto$ (this calculation is done explicitly in Lemma \ref{lem:fdd}). Moreover, by Proposition \ref{prop:auto}, the covariance matrix of $B_\pi$ is a symmetric, positive semi-definite matrix. 
This guarantees existence of the
Brownian motion $B_\pi$ and hence the pair $(B_\pi ,\BBB_\pi)$.
\end{rmk}
\par
For $\pi \in L(\CB,\reals^m)$ we define $\pi \otimes \pi \in L(\CB,\reals^{m\times m})$ by $(\pi\otimes \pi)^{ij} = \pi^i \otimes \pi^j$, where $\pi^i \otimes \pi^j$ is (as usual) the dual tensor product. 

\begin{lemma}\label{lem:fdd}
Let 
$\pi \in L (\CB , \reals^m)$ for some $m \in \naturals$.
As $\eps\to0$,
\begin{equ}
(\pi W_\eps , (\pi \otimes \pi) \BBW_\eps ) \to(B_\pi , \BBB_\pi)
\end{equ}
in the sense of finite dimensional distributions of stochastic processes.
\end{lemma}

\begin{proof} 
we have
\begin{equ}
(\pi W_\eps , (\pi \otimes \pi)\BBW_\eps)(t) = \bigg( \eps^{-1} \int_0^t (\pi b)(\yeps(r))dr ,  \eps^{-2} \int_0^t \int_0^r (\pi b)(\yeps(u))\otimes (\pi b)(\yeps(r)) dudr \bigg)\;.
\end{equ}
Now 
\begin{equs}
|(\pi b)(y) - (\pi b) (z)|  & = |\pi (b(\cdot,y)- b(\cdot,z))| \leq 
\norm{\pi}_{L(\CB,\R^m)}\norm{b(\cdot,y)-b(\cdot,z)}_{\CB}
\\ &  \leq \norm{\pi}_{L(\CB,\R^m)}\norm{b}_{C^{\theta,\kappa}} |y-z|^\kappa \;.
\end{equs}
Similarly, $|(\pi b)(y)|\leq
\norm{\pi}_{L(\CB,\R^m)}\norm{b}_{C^{\theta,0}}$.
Hence $\pi b\in C_0^\kappa(\Omega,\R^m)$ and the desired convergence 
follows from Proposition~\ref{prop-neps}.
\end{proof}

%
%

%

\begin{rmk}
Clearly, we can equally characterize the distribution of $(\pi_1 W , (\pi_2 \otimes \pi_3) \BBW)$ using this result, where each $\pi_i : \CB \to\reals^{m_i}$. Simply set $\pi = (\pi_1, \pi_2 , \pi_3)$ and then project out the unnecessary
components.  
\end{rmk}

\begin{rmk} \label{rmk-partial}
It would be natural to combine the tightness of  $\{\BW_\eps\}_{\eps > 0}$ with the convergence of finite dimensional distributions of $\BW_\eps$ obtained in Lemma \ref{lem:fdd} to obtain a weak limit theorem for $\{\BW_\eps\}_{\eps > 0}$. We avoid this here since showing that the finite dimensional distributions from Lemma \ref{lem:fdd} actually separate measures on $\CC^\gamma$ is a non-trivial task. Moreover, we gain nothing by doing so since, as shown in
Lemma~\ref{lem:rde_ito} below, all limit points $X$ agree.   
\end{rmk}

\section{Characterizing the RDE as a Diffusion}\label{s:rde_ito}

In this section we complete the proof of Theorem \ref{thm:fastslow_besov}.
The final ingredient is the following.

\begin{lemma}\label{lem:rde_ito}
Let $\BW$ be any limit point of $\{ \BW_\eps\}_{\eps > 0}$ and let $X$ be the solution to the RDE~\eqref{e:lemma_RDE} driven by $\BW$. Then $X$ is a weak solution to the SDE \eqref{e:sde}. 
\end{lemma}

\begin{proof}
Let $\vphi : \reals^d \to \reals$ be a smooth function and let $\gen$ be the generator of the SDE \eqref{e:sde}, given by
\begin{equ}
\gen \vphi (x) = \sum_{i=1}^d {\atilde}^i(x) \del_i \vphi (x) + \sum_{i,j=1}^d \frac{1}{2}(\sigma \sigma^T)^{ij}(x)\del^2_{ij}\vphi(x) 
\end{equ}
with
\begin{equ}
\atilde^i(x) = \int_\Omega a^i(x,y)d\mu(y)  + \sum_{k=1}^d \auto(b^k(x,\cdot) , \del_k b^i(x,\cdot)) 
\end{equ}
and
\begin{equ}
(\sigma \sigma^T)^{ij}(x) = \auto(b^i(x,\cdot) , b^j(x,\cdot)) + \auto(b^j(x,\cdot) , b^i(x,\cdot))\;.
\end{equ}
By Proposition \ref{prop:auto}(b), $\sigma \sigma^T(x)$ is symmetric positive semidefinite, for each $x\in\reals^d$. Thus, by \cite[Theorem 4.5.2]{varadhan06} (or more precisely, \cite[Theorem 5.3.3]{kurtz86}) it is sufficient to show that $X$ solves the martingale problem associated with $\gen$. 
Specifically, 
let $\{\CF_t\}_{t\geq 0}$ be the filtration generated by the random variable $\BW : [0,T] \to \CB \times (\CB \otimes \CB)$. 
We must show that 
\begin{equ}
\vphi (X(t)) - \vphi(X(s)) - \int_s^t \mathcal{\gen} \vphi (X(r)) dr
\end{equ}
is an $\CF_t$-martingale. 
\par
Since $\BW \in \CC^\gamma_g$ it follows from the chain rule for RDEs,
Proposition~\ref{prop:CR}, that
\begin{equs} \label{eq-CR}
\vphi(X(t)) = \vphi(X(s)) &+\sum_{i=1}^d \int_s^t \del_i \vphi(X(r))\int_\Omega a^i(X(r),y)d\mu(y) dr \\
&+ \sum_{i=1}^d\int_s^t \del_i \vphi(X)H^i(X)d\BW \;,
\end{equs}
with the equality holding pathwise, 
where we have used the identity 
$F(x)\bar a = \int_\Omega a(x,y)d\mu(y)$. 
Using \eqref{eq-CR} together with the ``divergence-form'' of $\gen$,
\begin{equ}
\gen \vphi (x) = \sum_{i=1}^d \int_\Omega a^i(x,y) d\mu(y)\del_i\vphi(x) 
+\sum_{i,k=1}^d \auto\big(b^k(x,\cdot),\del_k\{b^i(x,\cdot)\del_i\vphi(x)\}\big)\;,
\end{equ} 
we reduce to showing that
for each $i = 1,\dots,d$,
\begin{equ}\label{e:integ_claim}
\BE ( S | \CF_s )= \sum_{k=1}^d \BE \bigg( \int_s^t G_k(X(r)) dr \,\big|\, \CF_s \bigg) 
\end{equ}
for all $s\leq t \leq T$, where
\begin{equs}
& S  = 
\int_s^t \del_i \vphi(X)H^i(X)d\BW \;, \\
& G_k(x)  = 
\auto\big(b^k(x,\cdot),\del_k\{b^i(x,\cdot)\del_i\vphi(x)\}\big)\;.
\end{equs}
\par
By definition of the rough integral in~\eqref{e:dW_intS}, using
Proposition~\ref{prop-X'soln}, we see that
$S= \lim_{\Delta \to 0} S_\Delta$
where the limit is defined pathwise and 
\begin{equs}
S_{\Delta} = 
\sum_{[t_n,t_{n+1}] \in \Delta}  
&\del_i \vphi(X(t_{n}))H^i(X(t_{n}))W(t_{n},t_{n+1}) \\ &+ \big(H^k(X(t_{n})) \otimes  \del_k \{\del_i \vphi(X(t_{n}))H^i(X(t_{n}))\}\big) \BBW(t_{n},t_{n+1})
\end{equs} 
and $\Delta = \{ [t_n,t_{n+1}] : 0\leq n \leq N-1 \}$ denotes partitions of $[s,t]$. In the second term, we omit the sum over $k=1,\dots, d$ for brevity.


Next, we define 
\begin{equs}
M_\Delta = S_\Delta - \sum_{[t_n,t_{n+1}] \in \Delta}  G_k(X(t_n)) \Delta t_n\;,
\end{equs}
where $\Delta t_n=t_{n+1}-t_n$.  
It follows directly from the regularity of $b,X,\vphi$ that the map
\begin{equ}
t \mapsto \big( b^k(X(t),\cdot) , \del_k \{b^i(X(t),\cdot) \del_i \vphi(X(t))\} \big) \;\;,\;\; [0,T] \to C^\kappa(\Omega) \times C^\kappa(\Omega)
\end{equ}
is continuous.
By Proposition~\ref{prop:auto}(c), $t\mapsto G_k(X(t))$
is continuous and hence Riemann integrable.
%
In particular, $\lim_{\Delta\to0}(S_\Delta - M_\Delta)=
\int_s^t G_k(X(r))dr$
almost surely.  Hence, 
 \begin{equs}
 \BE (S | \CF_s ) & = \BE ( \lim_{\Delta \to 0} M_\Delta | \CF_s ) + \BE ( \lim_{\Delta \to 0} (S_\Delta - M_\Delta ) | \CF_s  ) \\
 & = \BE ( \lim_{\Delta \to 0} M_\Delta | \CF_s ) 
+ \BE \bigg( \int_s^t G_k(X(r)) dr  \,\big|\, \CF_s \bigg)\;.
 \end{equs}
\par
Thus proving~\eqref{e:integ_claim} reduces to showing that $\BE ( \lim_{\Delta \to 0} M_\Delta | \CF_s )  =0$. 
We claim that $M_\Delta$ is square integrable uniformly in $|\Delta|\le 1$ and that $\BE (M_\Delta | \CF_s) = 0$ for each $|\Delta| \leq 1$.   
Then
by convergence of first moments,
$\BE ( \lim_{|\Delta|\to0 }  M_\Delta | \CF_s)= \lim_{|\Delta|\to0 } \BE (M_\Delta | \CF_s) = 0$,
completing the proof.

It remains to verify the claim.
For each $x\in \reals^d$ let us define the projection $\pi(x)=(\pi_1(x),\pi_2(x),\pi_3(x)):
\CB\to\reals^3$ by
\begin{equ}
\pi_1(x) = \del_i \vphi(x)H^i(x) ,   \quad
\pi_2(x) = H^k(x) , \quad
\pi_3(x) =  \del_k \{\del_i \vphi(x)H^i(x)\}.
\end{equ}
Then
\begin{equ}
M_{\Delta} = \sum_{[t_n,t_{n+1}] \in \Delta}  
(M_\Delta^{n+1}- M_\Delta^n)
\end{equ}
where
\begin{equ}
M_\Delta^{n+1}=
M_\Delta^n+
\pi_1(X(t_n)) W(t_{n},t_{n+1}) + \big(\pi_2(X(t_n)) \otimes \pi_3(X(t_n))\big)  \BBW(t_{n},t_{n+1}) - G_k(X(t_n))\Delta t_n.
\end{equ} 
Note that 
\begin{equs} \label{e:pi}
& \sup_x\norm{\pi_1(x)b}_{C^\kappa}\lesssim \norm{b}_{C^{0,\kappa}}\;, \quad
\sup_x\norm{\pi_2(x)b}_{C^\kappa}\lesssim \norm{b}_{C^{0,\kappa}}\;, \\ &
\sup_x\norm{\pi_3(x)b}_{C^\kappa}\lesssim \norm{b}_{C^{1,\kappa}}\;.
\end{equs}
Also, by Lemma \ref{lem:fdd}, 
\begin{equ}
(\pi_1(x) W, (\pi_2(x) \otimes \pi_3(x)) \BBW) \edist (B_{\pi(x)}^1,\BBB_{\pi(x)}^{2,3}) 
\end{equ}
where $B_{\pi(x)} = (B_{\pi(x)}^1,B_{\pi(x)}^2,B_{\pi(x)}^3)$ is an $\CF^{\pi(x)}_t$-Brownian motion in $\reals^3$ with covariance defined by \eqref{e:Bpi_cov} and where $\BBB_{\pi(x)}$ is the corrected It\^o integral defined by \eqref{e:Bpi_ito}. 

Define the filtration $F_n = \CF_{t_n}$ for $n=0,\dots,N$
and set $F_n^\pi=\CF_{t_n}^{\pi(X(t_n))}$. 
Let $\BE^\pi_n(\cdot|F_n^\pi)$ denote conditional
expectations with respect to 
$\BP^{\pi(X(t_n))}$.
In particular, 
\begin{equ}
\BE(M_\Delta^{n+1} - M_\Delta^n|\CF_n)   =  \BE^\pi_n\big(B_{\pi(X(t_n))}^1(t_{n},t_{n+1})+ \BBB_{\pi(X(t_n))}^{2,3}(t_{n},t_{n+1})|F_n^\pi\big) 
-  \BE(G_k(X(t_n))\Delta t_n|\CF_n) \;.
\end{equ}
Since $B_{\pi(x)}$ and the It\^o integral part of $\BBB_{\pi(x)}$ are martingales for all $x\in\R^d$,
\begin{equs}
\BE ( M^{n+1}_\Delta - M^{n}_\Delta | F_n  ) &= \BE_n^\pi(\auto(\pi^2(X(t_n))b,\pi^3(X(t_n))b)\Delta t_n\,|\,F_n^\pi)-
\BE( G_k(X(t_n))\Delta t_n \,|\, F_n ) =0
\end{equs} 
where the last equality follows from the identity
\begin{equ}
\auto(\pi^2(x) b , \pi^3(x) b) = \auto \big( b^k(x,\cdot) , \del_k \{b^i(x,\cdot) \del_i \vphi(x)\} \big)=G_k(x)\;.
\end{equ}
Thus $M_{\Delta}^n$ is an $F_n$-martingale. Moreover, 
\begin{equs} 
 & \BE \big( (M_\Delta^{n+1} - M_\Delta^n)^2 | F_n\big)
 \\ 
& \quad
= \BE^\pi_n \bigg( \bigg(  B_{\pi(X(t_n))}^1(t_n,t_{n+1}) +  \int_{t_n}^{t_{n+1}} B_{\pi(X(t_n))}^2(t_n,r) dB_{\pi(X(t_n))}^3(r)   \bigg)^2 \,\bigg|\, F^\pi_n \bigg)\\
& \quad
\lesssim  \BE^\pi_n \big( (  B_{\pi(X(t_n))}^1(t_n,t_{n+1}))^2 | F^\pi_n \big) +  \BE^\pi_n \bigg( \bigg( \int_{t_n}^{t_{n+1}} B_{\pi(X(t_n))}^2(t_n,r) dB_{\pi(X(t_n))}^3(r)\bigg)^2 \,\bigg|\, F^\pi_n \bigg)\;.
\end{equs}
But by definition of $B_{\pi(x)}$ we have
\begin{equ}
\BE^\pi_n ( (B^1_{\pi(X(t_n))}(t_n,t_{n+1}))^2 | F^\pi_n ) = 2 \diffu( \pi_1(X(t_n)) b , \pi_1(X(t_n)) b ) \Delta t_n
\end{equ}
where we use the shorthand $\diffu(v,w) =\frac12( \auto(v,w) + \auto(w,v))$. 
By Proposition~\ref{prop:auto}(c) and~\eqref{e:pi},
\begin{equ}
\BE^\pi_n ( (B^1_{\pi(X(t_n))}(t_n,t_{n+1}))^2 | F^\pi_n ) \lesssim \sup_x \norm{\pi_1(x) b}_{C^\kappa(\Omega,\reals)}^2 \Delta t_n \lesssim\norm{b}_{C^{1,\kappa}}^2\Delta t_n\lesssim \Delta t_n.
\end{equ}
Also, by the It\^o isometry, again using
Proposition~\ref{prop:auto}(c),
\begin{equs}
& \BE^\pi_n  \bigg(\bigg(\int_{t_n}^{t_{n+1}} B_{\pi(X(t_n))}^2(t_n,r) dB_{\pi(X(t_n))}^3(r)\bigg)^2 | F^\pi_n \bigg)\\
& \qquad =  \big(\diffu(\pi_2(X(t_n)) b, \pi_3(X(t_n)) b) \big)^2 \int_{t_n}^{t_{n+1}} (r-t_n) dr \lesssim 
\norm{b}_{C^{1,\kappa}}^2(\Delta t_n)^2\lesssim |\Delta|\Delta t_n \;.
\end{equs}
It follows that for $|\Delta|\le1$,
\begin{equ}
 \BE \big( (M_\Delta^{n+1} - M_\Delta^n)^2 | F_n\big) \lesssim \Delta t_n\;.
\end{equ}
In particular, $\{ M_\Delta^n\}_{n=0}^N$ is an $L^2$-martingale, with $L^2$ norm bounded uniformly in $|\Delta| \leq 1$. 
Moreover $M_\Delta = M^N_\Delta$, so this completes the verification of the claim.
\end{proof}

\begin{proof}[Proof of Theorem \ref{thm:fastslow_besov}]
By Theorem \ref{thm:xeps_limit}, we see that $x_\eps \to_w X$ along subsubsequences where $X$ solves the RDE \eqref{e:lemma_RDE}. By Lemma \ref{lem:rde_ito}, $X$ is a weak solution to the SDE \eqref{e:sde}. In particular all subsequences converge to the same limit. The formula for $\auto(v,w)$ follows easily by taking $\BE_\mu$ in Assumption \ref{ass:clt} and applying Assumption \ref{ass:moments} to obtain convergence of the mean. This completes the proof. 
\end{proof}

\section{Localization}\label{s:localization}

In this section, we lift the localized convergence result Theorem \ref{thm:fastslow_besov} to the full convergence result Theorem \ref{thm:abstract}. 

Let $\eta_R : \reals^d \to [0,1]$ be a smooth cutoff function with
\begin{equ}
\eta_R(x) = \begin{cases} 1 \quad&\text{ for $|x| \le R$}\\ 0 \quad&\text{for $|x| \ge 2R$}
\end{cases}\;.
\end{equ}
Let $a,b$ satisfy the assumptions of Theorem \ref{thm:abstract} and define $a_R(x,y) = a(x,y) \eta_R(x) $ , $b_R(x,y) = b(x,y) \eta_R(x) $. Clearly $a_R, b_R$ satisfy all the requirements of Theorem \ref{thm:fastslow_besov}. In particular, if we let $x_{\eps,R}$ denote the solution to \eqref{e:fastslow} with $a,b$ replaced by $a_R,b_R$ then Theorem \ref{thm:fastslow_besov} states that $\xepsR \to_w X_R$ where $X_R$ satisfies the SDE \eqref{e:sde} with $a,b$ replaced with $a_R, b_R$. The following result is the final ingredient required to complete the localization argument. 
%
%
%
%

\begin{lemma}\label{lem:XR_lim}
Let $X_R$ be the It\^o diffusion defined by
\begin{equ}
dX_R = \ahat_R(X_R)dt + \sigma_R(X_R)dB \quad , \quad X_R(0) = \xi\;,
\end{equ}
where the drift and diffusion coefficients $\ahat_R:\reals^d\to\reals^d$ and $\sigma_R : \reals^d \to \reals^{d\times d}$ are given by
\begin{equ}
\ahat_R^i(x) = \int a_R^i(x,y)d\mu(y) +  \sum_{k=1}^d \auto(b_R^k(x,\cdot),\del_k b_R^i(x,\cdot))  \quad, \quad i=1,\dots,d\;, 
\end{equ}
\begin{equ}
(\sigma_R (x) \sigma_R^T(x))^{ij} = \auto(b_R^i(x,\cdot),b_R^j(x,\cdot)) + \auto(b_R^j(x,\cdot),b_R^i(x,\cdot)) \quad, \quad  
i,j =1, \dots, d \;.
 \end{equ}
%
%
Then $X_R \to_w X$ in the supnorm topology, as $R\to \infty$.
\end{lemma}

\begin{proof}
Firstly, it is clear that the martingale problem associated with $X$ is well posed. Indeed, from \cite[Theorem 6.3.4]{varadhan06}, it is sufficient to obtain the Lipschitz estimate
\begin{equ}\label{e:lip_well}
|\atilde(x) -\atilde(z)| + |(\sigma \sigma^T)(x) - (\sigma \sigma^T)(z)| \lesssim |x-z|\;.
\end{equ}
But this is immediate from Proposition~\ref{prop:auto}(c) and the regularity of
$a$ and $b$.
\par
By \cite[Theorem 11.1.4]{varadhan06}, to prove convergence it is sufficient to show that the coefficients ${\atilde}_R$ and $\sigma_R\sigma_R^T$ converge uniformly on compact sets to $\atilde$ and $\sigma \sigma^T$ respectively. But $\atilde_R(x) = \atilde(x)$ and $\sigma_R\sigma_R^T(x) = \sigma \sigma^T(x)$ for all $|x| \leq R$. Hence, by taking $R$ sufficiently large, convergence on compact sets is immediate. 
\end{proof}

\begin{proof}[Proof of Theorem \ref{thm:abstract}]
We now show that $\xeps \to_w X$ in the supnorm topology, as $\eps \to 0$. Fix a closed set $U \subset C([0,T], \reals^d)$. By the portmanteau lemma, it suffices to show that
\begin{equ}\label{e:final_port}
\limsup_{\eps \to 0} \mu (x_\eps \in U) \leq \BP (X \in U)\;.
\end{equ}
For $R>|\xi|$, we let
$x_{\eps,R}$ be the solution to \eqref{e:fastslow} with $a,b$ replaced by $a_R,b_R$. 
By uniqueness and continuity of solutions to ODEs,
for each fixed $\eps$, either $x_{\eps}(t) = x_{\eps,R}(t)$ for all $0\leq t\leq T$ or $\sup_{t\leq T} |x_{\eps,R}(t)| \geq R$. Thus we have
\begin{equ}
 \BPmnu (x_\eps \in U)  \leq  \BPmnu (x_{\eps,R} \in U) + \BPmnu \big( \sup_{t\leq T} |x_{\eps,R}(t)| \geq R \big) \;.
\end{equ}
But, since $a_R, b_R$ satisfy the requirements of Theorem \ref{thm:fastslow_besov}, for each fixed $R$ we have that $x_{\eps,R} \to_w X_R$ in the supnorm topology as $\eps\to0$. Since $x\mapsto \sup_{t\leq T} |x(\cdot)|$ is a continuous function in the supnorm topology, it follows from the portmanteau lemma that
\begin{equs}
 \limsup_{\eps \to 0} \BPmnu (x_\eps \in U) & \leq  \limsup_{\eps\to 0}\BPmnu (\xepsR \in U) + \limsup_{\eps \to 0}\BPmnu \big( \sup_{t\leq T} |x_{\eps,R}(t)| \geq R \big)\\
& \leq \BP (X_R \in U) + \BP \big( \sup_{t\leq T} |X_{R}(t)| \geq R \big)
\;.
\end{equs}
Taking $\limsup_{R\to\infty}$ on both sides and using Lemma \ref{lem:XR_lim} (and again the portmanteau lemma), 
\begin{equ}
\limsup_{\eps \to 0} \BPmnu (x_\eps \in U) \leq  \BP (X \in U) + \limsup_{R \to \infty }\BP \big( \sup_{t\leq T} |X_R(t)| \geq R \big)\;.\end{equ}

But $X_R$ solves the SDE \eqref{e:sde} with coefficients $\atilde_R$ and $\sigma_R$ that are, by an argument identical to~\eqref{e:lip_well}, Lipschitz and bounded. It follows from \cite[Theorem 2.4.4]{mao07} that
\begin{equ}\label{e:L1_XR}
\BE \sup_{t\leq T} |X_R(t)| \le K\;,
\end{equ}
 where $K$ depends only on $T, \xi$ and $\sup_{x\in\reals^d} (|\atilde_R(x)|\vee |\sigma_R(x)|)$. 
 By Proposition~\ref{prop:auto}(c),
 \begin{equs}
 |\sigma_R(x)|^2 &\leq \sum_{i,j=1}^d 2 |\auto(b^i_R(x,\cdot,),b^j_R(x,\cdot,))|\\ &\lesssim \sum_{i,j=1}^d  \norm{b^i_R(x,\cdot,)}_{C^\kappa}\norm{b^j_R(x,\cdot,)}_{C^\kappa} \lesssim \norm{b_R}^2_{C^{0,\kappa}} \lesssim \norm{b}^2_{C^{0,\kappa}}
 \end{equs}
 and we can similarly bound $\sup_x|\atilde_R(x)|$ uniformly in $R$. It follows that the constant $K$ in~\eqref{e:L1_XR} can be chosen uniformly in $R$. 
Thus 
\begin{equ}
\limsup_{R \to \infty }\BP \big( \sup_{t\leq T} |X_R(t)| \geq R \big) \leq  \limsup_{R \to \infty} \BE \sup_{t\leq T}|X_R(t)| / R = 0\;,
\end{equ}
which proves \eqref{e:final_port}. 
\end{proof}

\subsection*{Acknowledgements}

{
The research of DK was supported by ONR grant N00014-12-1-0257. IM was supported in part by the European Advanced Grant \emph{StochExtHomog} (ERC AdG 320977).
}

\bibliographystyle{./Martin}
\bibliography{./fastslow}

\end{document}